\documentclass[hidelinks,nohypdvips]{siamart}

\usepackage{amssymb,graphics,amsmath,amsopn,amstext,amsfonts,color,fancybox,hyperref,bm}
\usepackage[toc,page]{appendix}
\usepackage{longtable}
\usepackage{lipsum}
\usepackage{subfig,placeins} 
\usepackage{lipsum}
\usepackage{amsfonts}
\usepackage{epstopdf}
\usepackage{graphicx}
\usepackage{verbatim}
\usepackage{algorithmic}
\usepackage{multirow}
\usepackage{booktabs}
\usepackage{tikz}
\usetikzlibrary{patterns, arrows.meta}

\newcommand{\epc}{\hspace{1pc}}

\newcommand{\E}{{\rm I\!E}}

\newcommand{\II}{{\rm I\!I}}

\newcommand{\wh}{\widehat}

\newcommand{\IE}{{\rm I}\!{\rm E}}
\newcommand{\IP}{{\rm I}\!{\rm P}}
\newcommand{\argmin}{\operatornamewithlimits{\operatornamewithlimits{\mbox{argmin}}}}
\newcommand{\gap}{\vspace{0.1in}}

\def\begar{$$\begin{array}{lll}}
\def\endar{\end{array}$$}
\def\begarlab{\begin{equation} \begin{array}{lll} \label}
\def\endarlab{\end{array} \end{equation}}
\def\argmax{\text{argmax}}
\def\argmin{\text{argmin}}

\def\ds1{{\mathrm{1 \hspace{-2.6pt} I}}}

\def\calA{{\cal A}}
\def\calB{{\cal B}}

\def\calC{{\cal C}}

\def\calD{{\cal D}}

\def\calF{{\cal F}}

\def\calH{{\cal H}}

\def\calM{{\cal M}}
\def\calN{{\cal N}}

\def\calU{{\cal U}}

\def\calX{{\cal X}}

\def\calZ{{\cal Z}}

\def\bg1{{g_1}}

\def\Z1N{Z_1^N}

\def\X1N{X_1^{N}}

\def\l0{{\lambda_0'}}

\def\sign{\text{sign}}

\newtheorem{example}[theorem]{Example}
\newtheorem{remark}[theorem]{Remark}

\DeclareSymbolFont{mathx}{U}{mathx}{m}{n}
\DeclareMathAccent{\widecheck}{\mathalpha}{mathx}{"71}

\title{
Estimation of Individualized Decision Rules Based on \\
an Optimized Covariate-Dependent Equivalent of Random Outcomes}

\author{Zhengling Qi\thanks{Department of Decision Sciences,
The George Washington University, Washington DC 20052.
{\tt Email: qizhengling@gwu.edu}.
} \and Ying Cui\thanks{The Daniel J.\ Epstein
Department of Industrial and
Systems Engineering, University of Southern California, Los Angeles, CA 90089.
{\tt Emails: yingcui@usc.edu; jongship@usc.edu}.  The work of these two authors was based on research partially
supported by the U.S.\ National Science Foundation grant IIS--1632971.} \and
Yufeng Liu\thanks{Department of Statistics and Operations Research, University of North Carolina, Chapel Hill, NC
	27599.
	{\tt Email:  yfliu@email.unc.edu}. The work of the first and third author was based on research partially supported by the U.S.\ National Science Foundation grant IIS-1632951
	and National Institute of Health grant R01GM126550.} \and Jong-Shi Pang\footnotemark[2]}

\begin{document}
\maketitle

\begin{abstract}
Recent exploration of optimal individualized decision rules (IDRs) for patients in precision medicine has attracted a lot of attention
due to the heterogeneous responses of patients to different treatments.  In the existing literature of precision medicine, an
optimal IDR is defined as a decision function mapping from the patients' covariate space into the treatment space that maximizes
the expected outcome of each individual.  Motivated by the concept of Optimized Certainty Equivalent (OCE) introduced originally
in \cite{ben1986expected} that includes the popular conditional-value-of risk (CVaR) \cite{rockafellar2000optimization},
we propose a decision-rule based optimized covariates dependent equivalent (CDE) for individualized decision making problems.
Our proposed IDR-CDE broadens the existing expected-mean outcome framework in precision medicine and enriches the previous concept
of the OCE.  Under a functional margin description of the decision rule modeled by an indicator function as in the literature of
large-margin classifiers, we study the mathematical problem of estimating an optimal IDRs in two cases: in one case, an optimal solution
can be obtained ``explicitly'' that involves the implicit evaluation of an OCE; the other case requires the numerical solution of an
empirical minimization problem obtained by sampling the underlying distributions of the random variables involved.  A major challenge
of the latter optimization problem is that it involves a discontinuous objective function.  We show that, under a mild condition at the
population level of the model, the epigraphical formulation of this empirical optimization problem is a piecewise affine, thus
difference-of-convex (dc), constrained dc, thus nonconvex, program.  A simplified dc algorithm is employed to solve the resulting dc
program whose convergence to a new kind of stationary solutions is established.
Numerical experiments demonstrate that our overall approach outperforms existing methods in estimating optimal IDRs under
heavy-tail distributions of the data. In addition to providing a risk-based approach for individualized medical treatments,
which is new in the area of precision medicine,
the main contributions of this work in general include: the broadening of the concept of the OCE, the epigraphical description
of the empirical IDR-CDE minimization problem and its equivalent dc formulation, and the optimization of resulting piecewise affine
constrained dc program.

\end{abstract}

\begin{keywords}
Precision medicine, individualized decision making, conditional value-at-risk, optimized covariate dependent equivalent,
dc programming for discontinuous optimization
\end{keywords}

\begin{AMS}
62P10,	65K05, 90C26
\end{AMS}


\section{Introduction}
Most medical treatments are designed for ``average  patients". Due to the patients' heterogeneity, ``one size fits all" medical treatment strategies
can be very effective for some patients but not for others.  For example, a study of colon cancer \cite{tan2012kras} found that patients with
a surface protein called KRAS are more likely to respond to certain antibody treatments than those without the protein.
Thus exploration of precision medicine has recently gained a significant attention in scientific research.  Precision medicine is a medical
model that provides tailored health care for each specific patient, which has already demonstrated its success in saving
lives \cite{bissonnette2012infectious, kummar2015application}.  One of the main goals in precision medicine, from the data analytic perspective,
is to estimate the optimal individualized decision rules (IDRs) that can improve the outcome of each individual.

\subsection{Estimating optimal IDRs: the expected-outcome 
approach}
\label{sec: Existing Literature in Estimating Optimal IDRs}
An IDR is a decision rule that recommends treatments/actions to patients based on the information of their covariates.
Consider the data collected from a single-stage randomized clinical trial involving different treatments.
Before the trial, a patient's information $X$, such as blood pressure and past medicine history,  is recorded.
The enrolled patient will be randomly assigned to take a  treatment denoted by $A$. After the patient receiving the
treatment/action, the outcome $\calZ$ of the patient can be observed. Without loss of generality, we may assume that the larger
$\calZ$ indicates the better condition a patient is in.

Let $\IP$ be the probability distribution of the triplet $Y$ of random variables {\small$(X, A, \calZ)$} and
let $\IE$ be the associated expectation operator,
where $X$ is a random vector defined on the covariates space $\calX\subseteq \mathbb{R}^p$,
$A$ is a random variable defined on the finite treatment set $\calA$ and $\calZ $ is a scalar random variable
representing outcome.  The likelihood of $(X, A, \calZ)$ under $\IP$ is defined as $f_0(x)\,\pi(a \,|\, x)\,f_1(z \,|\, x, a)$,
where $f_0(x)$ is the probability density of $X$, $\pi(a \,|\, x)$ is the probability of patients being assigned treatment $a$
given $X = x$ and $f_1(z \,|\,  x, a)$ is the conditional probability density of $\calZ$ given covariates $X = x$ and treatment
$A = a$.  For the clinical trial study, the value of $\pi(a \,|\, x)$ is known; for the observational study, this value can be
estimated via various methods such as multinomial logistic regression. 

An IDR $d$ is defined as a mapping from the covariate space $\calX$ into the action space $\calA$.
We let $\calD$ be the class of all measurable functions mapping from $\calX$ into $\calA$;
that is, ${\cal D}$ is the class of all measurable IDRs.
For any IDR $d \in {\cal D}$, define $\IP^{\,d}$ to be the probability distribution under which treatment $A$ is decided by $d$.
Then the corresponding likelihood function under $\IP^{\,d}$ is $f_0(x)\,\II(a = d(x))\,f_1(z\,|\, x, a)$, where the
indicator function $\II(a = d(x))$ equals to $1$ if $a = d(x)$ and $0$ otherwise.  Note that this is a discontinuous
step function.  The expected-value function \cite{qian2011performance} based on $\IP^{\,d}$ is given as
$\E^{\,d\,}[\,\calZ \,]$,
which can be interpreted as the expected outcome under IDR $d$.  It is known that if $\pi(a \,|\, X) \geq a_0 > 0$
almost surely (a.s.) for any $a \in \calA$ and some constant $a_0$, then $\IP^{\,d}$ is absolutely continuous
with respect to $\IP$~\cite{qian2011performance}.  Thus by the Radon-Nikodym theorem,
\begin{equation}\label{value fun}
\E^{\,d\,}[\,\calZ \,] \, =\,  \E \left[\,\calZ\, \frac{\text{d} \IP^{\,d\,}}{\text{d} \IP}\,\right] \, = \, \E\left[\,\frac{\calZ \, \II(A = d(X))}{\pi(A | X)}\,\right].
\end{equation}
In particular,  $\E^{\,d\,}[\,c(X) \,] = \E[\,c(X) \,] $ for any integrable function $c$ of the covariate $X$ \cite{qian2011performance}.
Given the triplet $(X, A, \calZ)$, an optimal IDR under the expected-value function framework is defined as
\[
d_0 \, \in \, \underset{d \in \calD}{\argmax} \ \E^{\,d\,}[\,\calZ \,].
\]
This is the expected-value function
maximization approach to the problem of estimating an optimal IDR to date.  This approach
can be roughly categorized into two main types: model-based and classification-based methods.
One of the representative methods for the former approach is  Q-learning, which models the conditional mean of the outcome $\calZ$
given $X$ and $A$.  The treatment was then searched to yield  the largest conditional mean of
outcome \cite{watkins1992q,murphy2005generalization,qian2011performance,schulte2014q}.  Alternatively, the classification-based method,
which was first proposed in \cite{zhao2012estimating}, transforms the problem of maximizing $\E^{\,d\,}[\,\calZ \,]$ into minimizing
a weighted 0--1 loss.  Based on this transformation, various classification methods can be used to estimate the optimal
IDR \cite{laber2015tree,liu2016robust,zhou2017residual}.

Only maximizing the average of outcome under IDR $d$ may be restrictive in  precision medicine.  For example, when evaluating
several treatments' effects on patients, doctors may want to know which treatment does the best to improve the outcome of a
higher-risk patient.  More importantly, due to the complex decision-making procedure in precision medicine, an ``optimal" IDR that
only maximizes the expected outcome of patients may lead to potentially adverse consequences for some patients.  Therefore, considering
individualized risk exposure is essential in precision medicine.  This motivates us to examine the problem of determining optimal IDRs
under a broader concept to control the individualized risk of each patient.

\subsection{Optimized certainty equivalent}
\label{sec: Optimized Certainty Equivalent}
Estimating optimal IDRs can be regarded as an individualized decision-making problem.
Utility functions have played an important role in such problems since they characterize the preference order
over random variables, based on which decisions can be made.  Guarding against the hazard of adverse decisions, risk
measures are needed to balance the sole maximization of such utilities.  This bi-objective consideration is well
appreciated in portfolio management, leading to many risk measures since the early days of the mean-variance approach
in \cite{markowitz1952portfolio}.
We refer the readers to \cite{rockafellar2013fundamental} and references therein for a contemporary perspective of diverse
risk measures. 
Among such measures used in investment and economics, one of the most popular is the conditional-value-at-risk (CVaR) that has been extensively discussed
in \cite{rockafellar2000optimization,rockafellar2002conditional}; see the recent survey in \cite{SarykalinUryasev2008survey}.  In general,
for an essentially
bounded random variable ${\cal Z}$ with the property that there exists a large enough scalar $B > 0$ such that
the set $\left\{ \, \omega \in \Omega \, \mid \, | \, {\cal Z}(\omega) \, | \, > \, B \right\}$ has measure zero,
where $\Omega$ is the sample space on which the random variable ${\cal Z}$ is defined, the $\gamma$-CVaR of ${\cal Z}$
is by definition:
\[
\text{CVaR}_{\,\gamma\,}({\cal Z}) \, \triangleq \, \displaystyle{
	\sup_{\eta \in \mathbb{R}}
} \, \left[ \, \eta - \frac{1}{\gamma}\, \E\,  (\eta -{\cal Z})_+ \, \right],
\]
with $\gamma \in (0, 1)$ and $t_+ \triangleq \max(t,0)$ for a scalar (or vector) $t$.
The smallest maximizer of $\text{CVaR}_\gamma({\cal Z})$ is the $\gamma$-quantile of $\calZ$,
which is also known as the value-at-risk (VaR).  It turns out that the CVaR is a special case of an {\bf Optimized Certainty Equivalent} (OCE)
proposed in \cite{ben1986expected,ben1987penalty,ben2007old} that provides a link between utility and risk measures.
In fact, the introduction of the OCE predates the popularity of the CVaR in portfolio management.

Let ${\cal U}$ denote the family of utility functions $u : \mathbb{R} \to [ \, -\infty, \, \infty \, )$ that are
upper semi-continuous, 
concave, and non-decreasing with a nonempty effective domain
\[
\mbox{dom}(u) \, \triangleq \, \left\{ \, t \in \mathbb{R} \mid u(t) > -\infty \, \right\} \, \neq \, \emptyset
\]
such that $u(0) = 0$ and $1 \in \partial u(0)$, where $\partial u$ denotes the subdifferential map of $u$.  Thus in particular,
\[ 
\left[ \, u(t) \, \geq \, 0, \ \forall \, t \, \geq \, 0 \, \right] \epc \mbox{and} \epc 
\left[ \, u(t) \, \leq \, t, \ \forall \, t \, \in \, \mathbb{R} \, \right].
\]
The OCE of an essentially bounded random variable ${\cal Z}$ is by definition:
\[
{\cal O}_u({\cal Z}) \, \triangleq \, \displaystyle{
\sup_{\eta \in \mathbb{R}}
} \, \left[ \, \eta + \E\, u( {\cal Z} - \eta ) \, \right].
\]
According to the above cited references, the scalar $\eta$ is interpreted as the present consumption among the
uncertain future income ${\cal Z}$. Then  the sum $\eta + \E\, u( {\cal Z} - \eta )$ is the utility-based present
value of ${\cal Z}$. Thus the goal of the OCE is to maximize the latter value by choosing an optimal allocation
of ${\cal Z}$ between present and future consumption.  A particular interest of the OCE is the case where
$u(t) = \xi_1 \, \max(0,t) - \xi_2 \, \max(0,-t)$ for some constants
$\xi_1$ and $\xi_2$ satisfying $0 \leq \xi_1 \leq 1 \leq \xi_2$.  In this case, a maximizer of ${\cal O}_u(\calZ)$
corresponds to a quantile of the random variable $\calZ$. For $\xi_1 = 0$, ${\cal O}_u(\calZ)$ reduces to the CVaR. With a proper truncation,
a concave quadratic utility function can also satisfy the non-decreasing property, resulting in a mean-variance combination;
see \cite[Example~2.2]{ben2007old}. One special property of OCE is that $-{\cal O}_u(\calZ)$ gives a convex risk measure \cite[Section~2.2]{ben2007old}.
One of the limitations of the OCE, when applied to our problem of estimating optimal IDRs, is that it does not take into account covariates for the choice of an optimal allocation between present and
future consumption when data on the covariates are available.

In this paper, motivated by applications in the field of precision medicine, we {\bf Individualize} the known concept of the OCE
to a {\bf Decision-Rule based Optimized Covariate-Dependent Equivalent} (IDR-CDE)  that also incorporates domain covariates.
The new equivalent not only broadens the traditional expectation--only based  criterion in the estimation of the optimal IDRs
in precision medicine, but also enriches the combined concept of utility and risk measures and bring them to individual-based
decision making.  The proposed IDR-CDE is very flexible so that different utility functions will produce different optimal IDRs
for various purposes.  It turns out that estimating optimal IDRs under the IDR-CDE is a challenging optimization problem
since it involves the discontinuous function $\II(A = d(X))$.  A major contribution of our work is that we overcome this technical
difficulty by reformulating the estimation problem as a difference-of-convex (dc) constrained dc program under a mild assumption
at the population level of the model.  This reformulation allows us to employ a dc algorithm for solving the resulting dc program.
Numerical results under the settings of binary actions and linear decision rules are presented to demonstrate the performance
of our proposed model and algorithm.

\subsection{Contributions and organization}

The contributions of our paper are in two directions: modeling and optimization. In the area of modeling, we extend the
expected-value maximization approach in precision medicine to a more general framework by incorporating risk;
see Section~\ref{sec:main}.  This is accomplished through the extension of the OCE to the IDR-CDE in which we incorporate domain
covariates and individualized decision rules.  Properties of the IDR-CDE are derived in Subsection~\ref{subsec:properties}.
The optimal IDR problem under the IDR-CDE criterion is formally defined in Subsection~\ref{subsec:the IDE opt}.
Two cases of this problem are considered:
the decomposable case (Subsection~\ref{subsec:decomposable}) and the general case via empirical maximization.
Examples of the IDR-CDE given in Subsection \ref{subsec:examples} conclude the modeling part of the paper.  Beginning in Section~\ref{sec:empirical},
the solution of the empirical IDR-CDE maximization is the other major topic of our work.  The challenge of this problem is the presence of
the discontinuous indicator function in the objective function. The cornerstone of our treatment of this problem is its epigraphical formulation
which is valid under a mild assumption at the model's population level.  We next introduce a piecewise affine description of the epigraphical constraints from which
we obtain a difference-of-convex constrained optimization problem to be solved; see Sections \ref{sec:alg} and \ref{sec:dca}.
Although restricted to the empirical IDR-CDE maximization
problem, we believe that our novel dc constrained programming treatment of the discontinuous optimization problem on hand can potentially be generalized
to the composite optimization of univariate step functions with affine functions.
In Section~\ref{sec:Numerical}, we demonstrate the effectiveness of our proposed IDR-CDE optimization over the expected-value maximization
via numerical results.

\section{The IDR-based CDE}\label{sec:main}

In this section, we extend the OCE along two directions.
The first extension is to take the expectation $\E^{\,d\,}$ with respect to decision-rule based probability distribution $\IP^{\, d}$
in order to evaluate the outcome under the IDR $d$.
The second extension is to allow the deterministic scalar $\eta$ over which the supremum in the OCE is taken
to be a family of measurable functions $\calF$ defined on the covariate space $\calX$.  This family $\calF$  allows the incorporation of
available data representing covariate information for prediction and risk reduction; see the inequality \eqref{ineq:OCE} below.
For notational purpose, we let ${\cal L}^{\,r}(\calX, \Xi, \IP_X) $ be the class of all measurable functions $f$ such that
$\int \, |\, f(X) \, |^r \, d \, \IP_X \,< \,\infty$ with $r \in [1, \infty]$. Here $(\calX, \Xi, \IP_X)$ is the measure space
with $\Xi$ being the $\sigma$-algebra generated by $\calX$, and $\IP_X$ being the corresponding marginal probability measure
of $X$.  

\subsection{Definition and properties} \label{subsec:properties}
For an essentially bounded random variable ${\cal Z}$, 
the {\sl individualized decision-rule based optimized covariate-dependent equivalent} (IDR-CDE)
of ${\cal Z}$ under decision rule $d$ with respect to a utility function $u \in {\cal U}$ and a linear
space ${\cal F} \subseteq {\cal L}^{\,1}({\cal X}, \Xi, \IP_{X})$ is
\[ \begin{array}{lll}
{\cal O}_{(u, \calF)}^{\, d}({\cal Z}) & \triangleq & \displaystyle{
	\sup_{\alpha \in {\calF}}
} \, \left[ \, \E \,\alpha(X) + \E^{\,d\,} u( {\cal Z} - \alpha(X) ) \, \right] \\ [0.2in]
& = & \displaystyle{
	\sup_{\alpha \in {\calF}}
} \, \left[ \, \E \,\alpha(X) + \E\left( u( {\cal Z} - \alpha(X) )  \, \displaystyle{
	\frac{\II( A = d(X) )}{\pi(A | X)}
} \, \right) \, \right]  \\ [0.25in]
& = & \displaystyle{
	\sup_{\alpha \in {\calF}}
} \, \E\left[ \, \left[ \, \alpha(X) + u( {\cal Z} - \alpha(X) ) \, \right] \, \displaystyle{
	\frac{\II( A = d(X) )}{\pi(A | X)}
} \, \right],
\end{array} \]
where the last equality holds because of {$\E[\alpha(X)] = \E^d[\alpha(X)]$ and the change of measure}.
The space $\calF$ is taken to contain all constant functions and such that the expectations in
${\cal O}_{(u, \calF)}^{\, d}({\cal Z})$ are taken over integrable functions. One example of such a space is a
family of all bounded measurable functions.
We will specify $\calF$ for different utility functions in later discussion. The following proposition gives two preliminary properties
of the IDR-CDE.  In particular, the inequality (\ref{ineq:OCE}) bounds the IDR-CDE ${\cal O}_{(u, \calF)}^{\, d}({\cal Z})$ of the
random variable ${\cal Z}$ in terms of the OCE of ${\cal Z}$ in two ways: one is an upper bound in terms of the expected OCE of
${\cal Z}$ conditional on $X$ and $A = d(X)$, and the other one is a lower bound in terms of the decision-rule based OCE of ${\cal Z}$.
A notable mention of both bounds is that they are independent of the family ${\cal F}$; see (\ref{ineq:OCE}).

\begin{proposition}\label{prop: prem}
The following two statements hold.

\noindent (a) For any $u \in \calU$, one has ${\cal O}_{(u, \calF)}^{\, d}(0) = 0$.

\noindent (b) For any linear space $\calF$ containing all constant functions and
for which ${\cal O}_{(u, \calF)}^{\, d}({\cal Z})$ is finite,
\begin{equation}\label{ineq:OCE}
\E\,[\,{\cal O}_u({\cal Z} | X, A = d(X))\,] \, \geq \, {\cal O}_{(u, \calF)}^{\, d}({\cal Z}) \, \geq\,
\displaystyle{
\sup_{\eta\in\mathbb{R}}
} \, \E^{\, d}\left[ \, \eta + u( {\cal Z} - \eta ) \, \right].
\end{equation}
\end{proposition}

\begin{proof}
(a) Since $u \in \calU$, one has $u(t) \leq t$ and then
\[
{\cal O}_{(u, \calF)}^{\,d}(0) \, \leq\, \sup_{ \alpha \in \calF} \left\{\,\E\,[\,\alpha(X)\,] + \E^{\,d\,}[\,0 - \alpha(X)\,]\right\}
\, =  \, 0,
\]
where the last equality holds since $\E^{\,d\,}(\alpha(X)) = \E\left[\alpha(X)\right]$. Meanwhile, $u(0) = 0$ leads to
\[
{\cal O}_{(u, \calF)}^{\,d}(0) \, \geq\, \E\,[\,0\,] + \E^{\,d\,}[\,0 - 0\,] \, =  \, 0,
\]
since $0 \in \calF$.  Combining the two inequalities gives the statement that ${\cal O}_{(u, \calF)}^{\,d}(0) = 0$.

\gap

(b) We can write 
\begin{equation*}
\begin{aligned}
{\cal O}_{(u, \calF)}^d(\calZ) & = \,\sup_{ \alpha \in \calF} \left\{\,\E\left[\;\sum_{a \in \calA}\, \II(d(X) = a) \,\E\,
\left[\,\alpha(X) + u(\calZ - \alpha(X)) \mid X, A = a\,\right] \,\right] \,\right\}\\[0.1in]
& =\, \sup_{ \alpha \in \calF} \left\{\, \E\,[\,\E\,[\,\alpha(X) + u(\calZ - \alpha(X)) \mid X, A = d(X)\,]\,] \,\right\}\\[0.05in]
& =\, \sup_{ \alpha \in \calF} \left\{\, \E\,[\,\alpha(X) +\E\,[\,u(\calZ - \alpha(X)) \mid X, A = d(X)\,]\,] \,\right\}\\[0.05in]
& \leq \, \E\,\left[\,\sup_{ s \in \mathbb{R}} \left\{\,s + \E\,[\,u(\calZ - s) \mid X, A = d(X)\,]\,\right\}\,\right]\\[0.05in]
& = \,  \E\,[\, {\cal O}_u(\calZ \,|\, X, A = d(X)) \,],
\end{aligned}
\end{equation*}
where the inequality holds because for any $\alpha(X)$, we have
$\alpha(X) +\E\,[\,u(\calZ - \alpha(X)) \mid X, A = d(X)\,] \leq \displaystyle{
\sup_{ s \in \mathbb{R}}
} \, \left\{\,s + \E\,[\,u(\calZ - s) \mid X, A = d(X)\,]\,\right\}$.	
The right-hand inequality in (\ref{ineq:OCE}) holds because ${\cal F}$ contains all constant functions.
\end{proof}

Our proposed IDR-CDE measures the outcome $\calZ$ via the decision-rule based optimal allocation between the covariate-dependent present
value $\alpha(X)$ and the future gain $\calZ - \alpha(X)$ under the utility function $u$.  Unlike the original OCE, the allocation $\alpha(X)$
depends on the available covariate information $X$ such as environmental factors that can help to decide the optimal allocation. 
{Take linear regression as an example; if the response $\calZ$ can be predicted by the linear combination of covariates $X$, 
then covariates $X$ can explain some variability behind $\calZ$; this could result in the reduction in the variance of $\calZ$ given the 
information of $X$.}  Thus considering
the broader covariate-based allocation $\alpha(X)$ could improve the allocation and further reduce the risk.  This is also demonstrated via
Proposition~\ref{prop: prem}, by recalling that the negative of
the standard OCE is a risk measure; indeed inequality (\ref{ineq:OCE}) confirms that incorporating covariate information may lead to
a reduced risk measure.  Proposition~\ref{prop: exchange} provides sufficient conditions for equality to hold between the IDR-CDE and the
conditional OCE.


Note that ${\cal O}_u(\calZ \,|\, X, A = d(X))$ is a random variable; it is the original OCE corresponding to the random variable with distribution
being the conditional distribution of the random variable $\calZ$ given $X$ and $A = d(X)$.  Thus we may think of it 
as a conditional OCE.
The IDR-CDE preserves many properties of the standard OCE which can be found in \cite{ben2007old}.  The following are several of these properties.

\begin{proposition}\label{basic prop}
Given the two triplets $(X, A, \calZ)$ and $(d, u, \calF)$, the following properties hold:
\begin{itemize}
\item[(a)] \textbf{Shift Additivity}: for any essentially bounded random variable ${\cal Z}$ and
any measurable function $c \in \calF$ such that $c(X)$ is essentially bounded, 
${\cal O}_{(u, \calF)}^{\,d}(\calZ + c\,(X)) = {\cal O}_u^{\,d}(\calZ) + \E\,[\,c\,(X)\,]$;
in particular, ${\cal O}_{(u, \calF)}^{\,d}(c\,(X)) = \E\,[\,c\,(X)\,]$;
\item[(b)] \textbf{Consistency}: for any  measurable function $\widehat{c}$ defined over $\calX \times \calA$ such that
$\widehat{c}\,(X, A)$ is essentially bounded, 
${\cal O}_{(u, \calF)}^{\,d}(\widehat{c}\,(X, A)) = \E\,[\,\widehat{c}\,(X, d(X))\,]$;
\item[(c).] \textbf{Monotonicity}: for any two essentially bounded random variables $\calZ_1$ and $\calZ_2$ such that
$\calZ_1(\omega) \leq \calZ_2(\omega)$ for almost all $\omega\in \Omega$, ${\cal O}_{(u, \calF)}^{\,d}(\calZ_1) \leq {\cal O}_{(u, \calF)}^{\,d}(\calZ_2)$;
\item[(d).] \textbf{Concavity}: for any two essentially bounded random variables $\calZ_1$ and $\calZ_2$ and any $\lambda \in (0, 1)$,
\[
{\cal O}_{(u, \calF)}^{\,d}\left(\lambda \, \calZ_1 + (1 - \lambda)\, \calZ_2\right) \, \geq \,
\lambda\, {\cal O}_{(u, \calF)}^{\,d}(\calZ_1) + (1- \lambda)\,{\cal O}_u^{\,d}(\calZ_2).
\]
\end{itemize}
\end{proposition}
\begin{proof}
(a) 
We have
\[
\begin{array}{l}
{\cal O}_{(u, \calF)}^{\,d}(\calZ + c\,(X)) \\[0.05in]
\epc = \, \displaystyle\sup_{ \alpha \in \calF} \left\{\, \E\,[\,\alpha(X)\,] + \E^{\,d\,}[\, u(\calZ + c\,(X)- \alpha(X))\,]\, \right\}
\nonumber\\[0.05in]
\epc = \, \E\,[\,c\,(X)\,] + \displaystyle\sup_{ \alpha \in \calF} \left\{\,\E[\,\alpha(X)- c\,(X)\,] + \E^{\,d\,}[\,u(\calZ + c\,(X)- \alpha(X))\,] \,\right\} \nonumber \\[0.05in]
\epc = \, \E\,[\,c\,(X)\,]  + \displaystyle\sup_{ (\alpha - c) \in \calF} \left\{\, \E\,[\,(\alpha - c)(X)\,] + \E^{\,d\,}[\,u(\calZ - (\alpha - c)(X))\,] \,\right\}
\nonumber \\[0.1in]
\epc = \, \E\,[\,c\,(X)\,]  + {\cal O}_{(u, \calF)}^{\,d}(\calZ),
\end{array}
\]
where the third equality holds since $\calF$ is a linear space.

\gap

\noindent
(b) Since 
$u(t) \leq t$, we have
\[
{\cal O}_{(u, \calF)}^{\,d}(\calZ) \, \leq\, \sup_{ \alpha \in \calF} \left\{\,\E\,[\,\alpha(X)\,] + \E^{\,d\,}[\,\calZ - \alpha(X)\,]\right\}
\, =  \, 
\E^{\,d\,}[\, \calZ \, ], 
\]
where the equality holds because $\E^{\,d\,}[\,\alpha(X)\,] = \E\,[\,\alpha(X)\,]$ by the definition of $\IP^{\,d}$.
Therefore, if $\calZ = \hat{c}(X, A)$ is essentially bounded, then
	\begin{equation*}
	\begin{array}{lll}
	{\cal O}_{(u, \calF)}^{\,d}(\calZ) & \leq & \E^{\,d\,}[\,\widehat{c}\,(X, A)\,] \\ [0.1in]
	& = & \E\,\left[\,\displaystyle{
\frac{\widehat{c}\,(X, A) \; \II(d(X) = A)}{\pi(A | X)}
} \,\right] \\[0.2in]
	& = & \E\,\left[\,\displaystyle{
\frac{\widehat{c}\,(X, d(X)) \; \II(d(X) = A)}{\pi(A | X)}
} \,\right]	\, = \, \E\,\left[\, \widehat{c}\,(X, d(X))\, \right].
	\end{array}		
	\end{equation*}
	Since $u(0) = 0$, by the definition of the supreme in ${\cal O}_{(u, \calF)}^{\,d}$, we derive
	\begin{equation*}
	\begin{array}{lll}
	{\cal O}_{(u, \calF)}^{\,d}(\widehat{c}\,(X, A)) & \geq & \E\,[\,\widehat{c}\,(X, d(X))\,] + \E^{\,d\,}[\,u(\widehat{c}(X, A) - \widehat{c}\,(X, d(X))\,] \\[0.1in]
    & = & \E\,[\,\widehat{c}\,(X, d(X))\,] + \E^{\,d\,}\,[\,u(\widehat{c}\,(X, d(X)) - \widehat{c}\,(X, d(X)))\,] \\ [0.1in]
    & = & \E\,[\,\widehat{c}\,(X, d(X))\,].
	\end{array}
	\end{equation*}
	Thus, ${\cal O}_{(u, \calF)}^{\,d}(\widehat{c}\,(X, A)) = \E\,[\,\widehat{c}\,(X, d(X))\,]$.\\
	
\noindent
(c) If $\calZ_1 \leq \calZ_2$, then $\calZ_1 - \alpha(X) \leq \calZ_2 - \alpha(X)$ for $\alpha \in \calF$.
Since $u \in U_0$ is a non-decreasing utility function, it follows that
	\begin{equation*}
	\begin{aligned}
	{\cal O}_{(u, \calF)}^{\,d}(\calZ_1) &= \sup_{\alpha \in \calF} \left\{\E\,[\,\alpha(X)\,] + \E^{\,d\,}[\,u(\calZ_1 - \alpha(X))\,]\, \right\} \\[0.1in]
	&\leq \sup_{\alpha \in \calF} \left\{\,\E\,[\,\alpha(X)\,] + \E^{\,d\,}[\,u(\calZ_2 - \alpha(X)) \,]  \, \right\}
	\, =\, {\cal O}_{(u, \calF)}^{\,d}(\calZ_2).
	\end{aligned}
	\end{equation*}	
\noindent
(d) For any $\lambda \in (0, 1)$, denote a random variable $\calZ_\lambda \,\triangleq \, \lambda\, \calZ_1 + (1 - \lambda)\,\calZ_2$ and a measurable function $\alpha_\lambda(X) \,\triangleq \, \lambda \,\alpha_1(X) + (1-\lambda)\,\alpha_2(X)$. Clearly $\calZ_\lambda$ is essentially bounded and $\alpha_\lambda(X) \in \calF$.
 Then by the concavity of $u$, we have
	\begin{equation*}
	\begin{array}{l}
	\E\,[\,\alpha_\lambda (X)\,] + \E^{\,d\,}[\,u(\calZ_\lambda - \alpha_\lambda(X))\,] \, \geq \,
	\lambda \left(\,\E\,[\,\alpha_1(X)\,] + \E^{\,d\,}[\,u(\calZ_1 - \alpha_1(X))\,]\,\right) + \\ [0.1in]
	\hspace{2.2in} (1-\lambda)\left(\E\,[\,\alpha_2(X)\,] + \E^{\,d\,}[\,u(\calZ_2 - \alpha_2(X))\,]\,\right).
	\end{array}
	\end{equation*}
Taking supremum over $\alpha_1$ and $\alpha_2$ on both sides, we may derive the stated result.
\end{proof}

Properties (a) and (b) extend corresponding results of the original OCE \cite[Theorem~2.1]{ben2007old} from a constant $\eta$ to a measurable function
that depends on $X$ and $A$; properties (c) and (d) are essentially the same as those in \cite[Theorem~2.1]{ben2007old}.
These properties justify the use of the IDR-CDE in decision making.  Shift Additivity means if the outcome is shifted by
some function over covariates, the IDR-CDE measure is shifted by the average of this function.  Thus the IDR $d$ is invariant under such a shift.
Consistency means that to evaluate the IDR-CDR of a measurable function over $\calX \times \calA$
is equivalent to evaluating the expectation of this random function when the action follows the decision rule $d$.
Monotonicity and concavity have the same respective meanings as the OCE: the former guarantees a larger CDE for a (stochastically) larger outcome;
the latter ensures that the IDR-CDE of a convex combination of two outcomes given a decision rule $d$
is always better than only considering each single outcome separately; this property encourages the simultaneous combination of multiple outcomes
for better results.
\subsection{The IDR optimization problem} \label{subsec:the IDE opt}
We employ the IDR-CDE to evaluate the decision rule $d$ of the outcome $\calZ$ via its optimized covariate equivalent,
with the goal of estimating an optimal IDR that maximizes the IDR-CDE given the pair $(u,\calF)$ in the following sense.
\begin{definition}\rm
Given the triplet $(X, A, \calZ)$, the pair $(u,\calF)$, and the family ${\cal D}$ of decision rules, an optimal IDR is a
rule $d^*$ such that
\[
d^\ast(X) \, \in \, \operatornamewithlimits{argmax}_{d \in \calD} \ {\cal O}^{\,d}_{(u,\calF)}(\calZ),
\]
if such a maximizer exists.  \hfill $\Box$
\end{definition}
\noindent
Thus we can compute $d^\ast(X)$ and the optimal allocation $\alpha^\ast(X)$ jointly by solving
\begin{equation}\label{eq: goal}
\sup_{d \in \calD, \alpha \in \calF} \E\,[\,\alpha(X)\,] + \E^{\,d\,}[\,u(\calZ - \alpha(X))\,].
\end{equation}
The rest of the paper is devoted to the solution of this optimization problem.  The discussion is divided into two cases
depending on whether we can exchange the supremum over $\alpha$ and the expectation $\E^{\, d}$ in ${\cal O}^{\,d}_{(u,\calF)}(\calZ)$.
The exchangeable case requires the theory of decomposable space from variational analysis; this leads to an ``explicit'' determination
of the optimal IDR via the evaluation of the conditional OCE given the covariate $X$ and the finite actions $a \in {\cal A}$;
see Proposition~\ref{thm: optimal IDR}.  The general case requires the numerical solution of an empirical optimization problem obtained
from sampling of the covariates among available data.
\subsection{Decomposable space and normal integrand} \label{subsec:decomposable}

In order to exchange the supreme over $\alpha(X)$ and expectation with respect to $\E^{\, d}$, we need to first introduce the concept
of a decomposable space and the normal integrand. 
\begin{definition} \rm
\cite[Definitions~14.59 and 14.27]{rockafellar2009variational}.
A space $\calM$ of ${\cal B}_0$-measurable functions is {\sl decomposable} relative to an underlying measure space $(\Omega_0, \calB_0, \mu)$
if for every function $x_0 \in \calM$, every set $G \in \calB_0$ with $\mu(G) < \infty$ and any bounded,
measurable function $x_1$, the function $x_2(t) = x_0(t)\II(t \not\in G) + x_1(t) \II(t \in G)$ belongs to $\calM$.
%
An extended-value function $f: \Omega_0 \times \mathbb{R} \rightarrow (-\infty, \infty]$ is a {\sl normal integrand} if its
epigraphical mapping
$\omega \rightarrow \mbox{epi } f(\omega,\cdot)$ is closed-valued and measurable.  \hfill $\Box$
\end{definition}

The space ${\cal L}^{\,r}(\calX, \Xi, \IP_{\cal X})$ is decomposable for $r \in [ 1, \infty ]$ but the family of constant functions
is not decomposable.
These facts will be used in the examples to be discussed in the next subsection.

We will employ the following simplified version of \cite[Theorem~14.60]{rockafellar2009variational} that provides the
required conditions for the exchange of the supremum and expectation in our context.

\begin{theorem} \label{th:normal integrand}
Let ($\Omega_0$, $\calB_0$, $\mu$) be a probability measure space, and $\calM$ be a decomposable space of
$\calB_0$-measurable functions.
Let $f: \Omega_0 \times \mathbb{R} \rightarrow (-\infty, \infty]$ be a normal integrand; let the integral functional
$I_f(x) = \int_{\Omega_0}f(x(\omega), \omega) d\mu(\omega)$ be defined on $\calM$.  The following two statements hold:

(a) $\displaystyle{
\inf_{x \in \calM}
} \, \int_{\Omega_0}f(x(\omega), \omega) d\mu(\omega) = \int_{\Omega_0} \displaystyle{
\inf_{s \in \mathbb{R}}
} \, f(s, \omega) d\mu(\omega)$ as long as $I_f(x)$ is finite; and

(b) $x_0 \in \underset{x \in \calM}{\argmin}\,  I_f(x) \Longleftrightarrow x_0(\omega) \in \underset{s \in \mathbb{R}}{\argmin}\, f(s, \omega)$
almost surely.  \hfill $\Box$
\end{theorem}
\noindent
The following proposition shows that if ${\cal F}$ is decomposable, then equality holds between the IDR-CDE and the conditional OCE.

\begin{proposition} 	\label{prop: exchange}
If $\calF$ is a decomposable space relative to $(\calX, \Xi, \IP_X)$, then
\[
{\cal O}_{(u, \calF)}^{\,d}(\calZ) = \E\,[\,{\cal O}_u(\calZ \,|\, X, A = d(X)) \,].
\]
\end{proposition}
\begin{proof}
Note that $\E\,[\,\alpha(X) + u(\calZ - \alpha(X)) \mid X, A = d(X)\,]$ is measurable with respect to $X$ and upper semi-continuous
with respect to $\alpha(X)$ for any $X$, 
thus
is a normal integrand \cite[Example~14.31]{rockafellar2009variational}.  Hence we have
\begin{equation*}\label{normal integrand}
\begin{aligned}
{\cal O}_{(u, \calF)}^d(\calZ) & =\, \sup_{ \alpha \in \calF} \left\{\, \E\,[\,\E\,[\,\alpha(X) + u(\calZ - \alpha(X)) \mid X, A = d(X)\,]\,] \,\right\}\\
& =\, \E\,\left[\,\sup_{ s \in \mathbb{R}} \left\{\,s + \E\,[\,u(\calZ - s) \mid X, A = d(X)\,]\,\right\}\,\right] \\[0.05in]
& = \,  \E\,[\, {\cal O}_u(\calZ \,|\, X, A = d(X)) \,],
\end{aligned}
\end{equation*}
where the second equality is by Theorem~\ref{th:normal integrand} because $\calF$ is decomposable and $\calZ$ is bounded.
\end{proof}
\begin{remark}\rm
Since the conditional OCE is independent of the space $\calF$, it follows that so is ${\cal O}_{(u, \calF)}^d(\calZ)$
provided that $\calF$ is decomposable relative to $(\calX, \Xi, \IP_X)$.
Thus, in the following, if we specify $\calF$ to be decomposable, then we omit $\calF$ and write the IDR-CDE of the random variable $\calZ$ as
${\cal O}_{u}^d(\calZ)$.  \hfill $\Box$
\end{remark}

As a result of Proposition \ref{prop: exchange}, we can characterize the optimal IDR explicitly if $\calF$ is a decomposable space.
We recall that ${\cal A}$ is a finite set.

\begin{proposition}\label{thm: optimal IDR}
	For a given decomposable space $\calF$ and utility function $u \in \calU$, an optimal IDR is given by
	\begin{equation} \label{eq:optimal IDR decomposable}
	d^{\,\ast}(X) \in \operatornamewithlimits{argmax}_{a \in \calA} \ {\cal O}_u(\calZ \,|\, X, A = a).
	\end{equation}
\end{proposition}
\begin{proof}
By the definition of ${\cal O}^{\,d}_u(\calZ)$, we have for any $d \in {\cal D}$,
\[ \begin{array}{lll}
\E\,\left[\,{\cal O}_u(\calZ \,|\, X, A = d(X)) \,\right]
& = & \E\,\left[\,\displaystyle{
\sum_{a \in \calA}
} \, \II(d(X) = a) \, {\cal O}_u(\calZ \, | \, X, A = a) \, \right] \\ [0.2in]
& \leq & \E \, \left[ \, \displaystyle{
\sum_{a \in \calA}
} \, \II(d(X) = a) \, \displaystyle{
\max_{a^{\prime} \in {\cal A}}
} \, {\cal O}_u(\calZ \, | \, X, A = a^{\prime}) \, \right] \\ [0.2in]
& = & \E \, \left[ \, \left( \, \displaystyle{
\max_{a^{\prime} \in {\cal A}}
} \, {\cal O}_u(\calZ \, | \, X, A = a^{\prime}) \, \right) \, \displaystyle{
\sum_{a \in \calA}
} \,  \II(d(X) = a) \, \right] \\ [0.2in]
& = & \E \, \left[ \, \displaystyle{
\max_{a^{\prime} \in {\cal A}}
} \, {\cal O}_u(\calZ \, | \, X, A = a^{\prime}) \, \right].
\end{array} \]
Therefore if (\ref{eq:optimal IDR decomposable}) holds, then $d^{\, \ast}$ is maximizing.  Such a $d^{\, \ast}$ is a measurable function
because being an optimal IDR,
$d^{\, \ast}(X) = a$ if and only if ${\cal O}_u(\calZ \,|\, X, A = a) \geq \displaystyle{
\max_{a^{\, \prime} \neq a}
} \, {\cal O}_u(\calZ \,|\, X, A = a^{\, \prime})$ and
${\cal O}_u(\calZ \,|\, X, A = a) \geq \displaystyle{
\max_{a^{\, \prime} \neq a}
} \, {\cal O}_u(\calZ \,|\, X, A = a^{\, \prime})$ is a measurable set with respect to $X$.
\end{proof}

\begin{remark}\rm
The explicit expression of an optimal IDR is valid only when the space $\calF$ is decomposable.
{If the conditional distribution of $\calZ$
given $X$ and $A = a$ is known, then it is possible
to compute the individualized OCE ${\cal O}_u(\calZ \,|\, X, A = a)$ directly.
For example, if we make certain parametric assumptions on this conditional distribution,
we may be able to estimate these parameters based on the collected data and obtain
optimal IDRs based on Proposition~\ref{thm: optimal IDR}.  This is similar to the model-based methods in the literature
of the expected-value function maximization approach.  However, the empirical performance could be affected by the possible model misspecification. Therefore the individualized OCE ${\cal O}_u(\calZ \,|\, X, A = a)$ is primarily a conceptual notion
and the expression (\ref{eq:optimal IDR decomposable}) is mainly for interpretation.}
\hfill $\Box$
\end{remark}

According to Proposition \ref{thm: optimal IDR}, an optimal IDR under our proposed CDE can be obtained by choosing the decision rule with the
largest individualized OCE.  In the next subsection, we will characterize the IDR-OCE via several illustrative examples for both decomposable
and non-decomposable families of covariate functions.
	
\subsection{Illustrative examples} \label{subsec:examples}

We present several common utility functions to further explain the IDR-CDE for individualized decision making.
We will focus on two families: ${\cal L}^{\, r}(\calX, \Xi, \IP_{X})$ for some $r \in [1, \infty]$ and a family of constant function which we denote
$\calF_c$.  The former family is a decomposable linear space and the latter family is not decomposable.

\begin{example}[Identity utility function]\label{ex 1}\rm
Let $u(t) = t$, then by the definition, we can obtain ${\cal O}_u^{\,d}(\calZ) = \E^{\,d\,}[\,\calZ\,]$ for both families $L^1(\calX, \Xi, \IP_{X})$ and $\calF_c$.
This recovers the expected-value maximization framework in the existing literature of precision medicine.  By Proposition~\ref{thm: optimal IDR},
for the family ${\cal L}^{\, r}(\calX, \Xi, \IP_{X})$, an optimal IDR under the identity utility function is given by:
\[
d^{\,\ast}(X) \in \displaystyle\operatornamewithlimits{argmax}_{a \in \calA} \E\,[\,\calZ \,|\, X, A = a \,] \, ,
\]
which is equivalent to  the action with the largest expected outcome $\calZ$ among all the actions given covariates $X$.
\hfill $\Box$
\end{example}

\begin{example}[Piecewise Linear Utility Function]\label{ex 2}\rm
Let \[u(t) = \xi_1 \max(0, t) - \xi_2 \max(0, -t),\epc\mbox{where} \;\, 0 \leq \xi_1 < 1 < \xi_2.\] It can be verified that $u \in U_0$.

\gap

\noindent
{\bf (a) Decomposable space: $\calF = L^{\,1}(\calX, \Xi, \IP_{X})$}. The  corresponding IDR-CDE is
\begin{equation} \label{CVaR}
{\cal O}_u^{\,d}(\calZ) = \sup_{ \alpha \in \calF} \left\{
\begin{array}{ll}
\,\E\,[\,\alpha(X)\,] +  \\[0.1in]
 \E^{\,d\,}[\,\xi_1 \max\,(\,0\, ,\,  \calZ - \alpha(X)\,) - \xi_2 \max\,(\,0\, , \, \alpha(X) - \calZ\,)\,] \,
\end{array}
\right\}.
\end{equation}
Based on Proposition~\ref{prop: exchange} we can write it as: with $\gamma \triangleq \displaystyle{
\frac{1 - \xi_1}{\xi_2-\xi_1}
}$.  Then the ${\cal O}_u^{\,d}(\calZ)$ is equal to
\[
\begin{array}{l}
\E\,\left[\,\sup_{ s \in \mathbb{R}} \big\{\, s + \E\,[\,\xi_1\,\max\,(\,0, \calZ - s\,) -
\xi_2\,\max\,(\,0\,,\, s - \calZ\,) \mid X, A = d(X)\,]\,\big\}\,\right]\\[0.1in]
= \, \xi_1\E^{\,d\,}[\, \calZ\,] + (1-\xi_1)\,\E\,\left[\, \sup_{ s \in \mathbb{R}} \left\{\, s  - \frac{1}{\gamma}\, \E\,[
\,\max\,(\,0\,,\, s - \calZ\,) \mid  X, A = d(X)\,]\right\}\,\right] \\[0.1in]
= \, \xi_1\, \E^{\,d\,}[\, \calZ\,] + (1-\xi_1)\,\E\,[\, \text{CVaR}_{\,\gamma\,}(\calZ \mid X, A = d(X))\,],
\end{array}
\]
where given $X$, the corresponding supremum is attained at the $\gamma$-quantile of conditional distribution of $\calZ$ on $X$ and $A = d(X)$ almost surely.
Therefore, under the piecewise affine utility function, ${\cal O}_u^{\,d}(\calZ)$ can be interpreted as a convex combination of the expected value of $\calZ$
and its expected CVaR given IDR $d$.  Thus this ${\cal O}_u^{\,d}(\calZ)$ considers both $\E^{\,d\,}[\, \calZ\,]$ and CVaR of the outcome simultaneously.
In particular, when $\xi_1 = \xi_2 = 1$, this recovers Example~\ref{ex 1}.

By Proposition \ref{thm: optimal IDR}, a corresponding optimal IDR is
\begin{equation} \label{ex 2 a}
d^{\,\ast}(X) \in \displaystyle\operatornamewithlimits{argmax}_{a \in \calA}\, \big\{\, \xi_1 \,\E\,[\,\calZ \,|\, X, A= a\,] +
(1-\xi_1)\, \text{CVaR}_{\,\gamma}(\calZ \,|\, X, A = a)\,\big\}.
\end{equation}
Therefore, under this piecewise affine utility function, an optimal IDE is to choose the action with the largest convex combination of
expected outcome and CVaR of outcome $\calZ$ among all the actions given covariates $X$.  \hfill $\Box$

\gap

\noindent
{\bf (b) Family of constant functions: $\calF = \calF_c$}.
The IDE-CDE reduces to \cite[Example~2.3]{ben2007old} with IDR $d$ involved:
\[
\begin{aligned}
{\cal O}_{(u, \calF_c)}^{\,d}(\calZ) &= \, \sup_{c \in \mathbb{R}} \left\{\, c + \E^{\,d\,}[\,\xi_1 \, \max\,(\,0\,,\, c - \calZ\,) - \xi_2\,  \max\,(\,0\,,\, c - \calZ\,)\,]    \, \right\}\\[0.05in]
& = \, \xi_1 \, \E^{\,d\,}[\,\calZ\,] + (1-\xi_1)\,\sup_{c \in \mathbb{R}}\left\{\,c - \frac{\xi_2-\xi_1}{1-\xi_1}\,\E^{\,d\,}[\,\max\,(\,0\, , \, c- \calZ\,)\,]     \,\right\}.
\end{aligned}
\]
The supremum in the right-hand side is any $c^\ast$ satisfying $\IP^{\,d\,}(\calZ \leq c^\ast) \geq \gamma$ and $\IP^{\,d\,}(\calZ \geq c^\ast) \leq 1-\gamma$, which
is the $\gamma$-quantile of ${\cal Z}$ under the probability distribution $\IP^{\,d}$, denoted by $Q^d_\gamma(\calZ)$.
The corresponding maximum value is $\xi_1\, \E^{\,d\,}[\,\calZ\,] + (1-\xi_1)\,\text{CVaR}^d_{\,\gamma}(\calZ)$.
By definition, an optimal IDR under $\calF_c$ is given by
\[
d^\ast \in \displaystyle\operatornamewithlimits{argmax}_d\,\left\{\xi_1\E^{\,d\,}[\,\calZ\,] + (1-\xi_1)\,\text{CVaR}^{\,d}_\gamma(\calZ)\,\right\}.
\]
While this expression is insightful, the above optimal IDR $d^\ast$ does not have an explicit form as \eqref{ex 2 a}
since Proposition \ref{thm: optimal IDR} no longer holds by the fact that $\calF_c$ is not a decomposable space.
\hfill $\Box$
\end{example}

\begin{example}[Quadratic Utility]\label{ex 3}\rm
Let
\[
u(t) \, = \, \left\{ \begin{array}{ll}
t - \displaystyle{
\frac{1}{2\tau}
} \, t^2 & \mbox{if $t \, \leq \, \tau$} \\ [0.1in]
\tau/2 & \mbox{otherwise},
\end{array} \right\}, \ \mbox{where} \ \tau \, = \, \displaystyle{
\sup_{\omega \in \Omega}
} \, \calZ(\omega) - \displaystyle{
\inf_{\omega \in \Omega}
} \, \calZ(\omega),
\]
be a quadratic function truncated to be an admissible utility function in the family ${\cal U}$
and to adopt to the range of the random outcome ${\cal Z}$.
Note that $u$ is continuously differentiable with derivative $u^{\, \prime}(t) = \left( \, 1 - \displaystyle{
\frac{t}{\tau}
} \, \right) \, \II(t \leq \tau)$.

\gap
\noindent
{\bf (a) Decomposable space: $\calF = {\cal L}^{\,2}(\calX, \Xi, \IP_{X})$}.
By Proposition \ref{prop: exchange}, we have,
\[
\begin{array}{ll}
{\cal O}_u^{\,d}(\calZ) & = \, \E \, \left[ \, \displaystyle{
\sup_{s \in \mathbb{R}}
} \, \left\{ \, s + \E \, \left[ \, u({\cal Z} - s) \, \mid \, X, A = d(X) \, \right] \, \right\} \, \right] \\ [0.2in]
& = \, \E^{\,d\,}[\,\calZ\,] - \E\,\left[\, \displaystyle{
\frac{1}{2\tau}
} \,\E\,\left[\,\left(\,\calZ - \E\,[\,\calZ \,|\, X, A = d(X)\,]\,\right)^2 \,|\, X, A = d(X)\,\right]\,\right] \\ [0.2in]
& = \, \E^{\,d\,}[\,\calZ\,]- \displaystyle{
\frac{1}{2\tau}
} \, \E\,\left[\,\text{var}(\calZ \, |\,  X, A = d(X))\,\right],
\end{array}
\]
where the supreme $\alpha^\ast(X) = \E\,[\,\calZ \,|\, X, A = d(X)\,]\;$ almost surely and $\text{var}(\bullet)$ is the variance of a random variable.
The second equality is based on \cite[Remark 2.1]{ben2007old} by  noting that $1 + \IE\left[ u^{\, \prime}({\cal Z} - \alpha^*(X)) \right] = 0$.
The interchange between expectation and derivative is justified by the dominated convergence theorem under the restriction that
$s \in \left[ \, \displaystyle{
\inf_{\omega \in \Omega}
} \, {\cal Z}(\omega), \, \displaystyle{
\sup_{\omega \in \Omega}
} \, {\cal Z}(\omega) \, \right]$.
Thus ${\cal O}_u^{\,d}(\calZ)$ can be interpreted as the (individualized) mean-variance risk measure under the decision rule $d$, generalizing the
mean-variance criterion in the absence of $A$ and $X$, which is frequently used in portfolio selection.
An optimal IDR is given by
\[
d^{\,\ast}(X) \,\in  \displaystyle\operatornamewithlimits{argmax}_{a \in \calA}\, \left\{ \, \E\,[\,\calZ \,|\, X, A=a\,] -
\displaystyle{
\frac{1}{2\tau}
} \, \text{var}\,[\, \calZ \, |\,  X, A = a\, ]\, \right\},
\]
which suggests the optimal action to maximize the expected outcome balanced with the variance given covariates $X$.

\gap

\noindent
{\bf (b) Family of constant functions: $\calF = \calF_c$}.
Similar to part (a) above, direct computation yields $	{\cal O}_u^{\,d}(\calZ) = \E^{\,d\,}[\,\calZ\,] - \displaystyle{
\frac{1}{2\tau}
} \, \text{var}^{\, d}(Z)$ with
$c^\ast = \E^{\, d}[\calZ]$, where $\text{var}^{\, d}(Z)$ denotes the variance of a random variable $\calZ$ under $\IP^{\, d}$.
An optimal IDR under $\calF_c$ is
\[
\displaystyle\operatornamewithlimits{argmax}_d \left\{\E^{\,d\,}[\,\calZ\,]-\frac{1}{2\tau}\text{var}^{\,d} \left(Z\right)\right\},
\]
which requires further evaluation by a numerical procedure.  \hfill $\Box$
\end{example}

From Example~\ref{ex 2} and Example~\ref{ex 3}, we see that one of the differences between a covariate-dependent $\alpha(X)$
and a constant $\alpha(X) \in {\cal F}_c$
lies in that for the former, the IDR-CDE considers expected individualized OCE given the decision rule $d$, but for a
constant $\alpha$, in contrast,
the IDR-CDE considers only the OCE of the random variable $\calZ$ under $\IP^{\, d}$.
{To further understand this difference, consider a toy example with $\calZ = X_1A + \varepsilon$,
where both $X_1$ and $\varepsilon$ independently follow the standard normal distribution.  Suppose we use the utility function
in Example~\ref{ex 2}(b) with $\xi_1 = 0$ and $\xi_2 = 2$ to evaluate an IDR $d(X_1) = 1$. By calculation, $c^\ast = 0$ and
thus we are focused on the median of $\calZ$ under the probability distribution $\IP^{\, d}$.  The corresponding
${\cal O}_{(u, \calF_c)}^{\,d}(\calZ) = \E[\E[\calZ\II(\calZ \leq 0)| X, A = 1]]$.  If we have one patient with covariate $X_1 = -2$,
then $\IP(\calZ \leq 0 | X_1 = -2, A = 1) \approx 84\%$. For this patient, ${\cal O}_{(u, \calF_c)}^{\,d}(\calZ)$ evaluates the
outcome lower than about $84\%$-quantile, which is not satisfactory.}  As a result, we may conclude that the optimal IDR cannot be
quantified by comparing each action separately of each other when considering $\alpha(X)$ being constant functions only.
Consequently such an IDR cannot control the individualized OCE.

{
So far we only consider single-stage individualized decision making problems.  It is also meaningful to extend our proposed IDR-CDE to
multi-stage decision-making scenarios in order to deliver time-varying optimal IDRs with risk exposure control.
Since it will require advanced modeling and treatment, we leave such an extension for future research.}

\section{The Empirical IDR Optimization Problem} \label{sec:empirical}
\label{sec:alg}
In this section, we discuss how to numerically solve the optimization problem \eqref{eq: goal} at the empirical level {without assuming any data generating mechanisms}.
In the following, we focus on estimating the optimal IDR with $\calA = \{-1, 1\}$, i.e., a binary action space.
Further, for computational purposes, we restrict the decision rule to be given by: $d( X ) = \mbox{sign}(f( X;\theta ))$
for a parametric linear estimation function:
$f( X;\theta ) = \beta^{\, T} X + \beta_0 = \theta^{\, T} \wh{X}$, where $\theta \triangleq \left( \begin{array}{l}
\beta \\
\beta_0
\end{array} \right) \in \mathbb{R}^{p+1}$
contains the unknown coefficients to be estimated and $\wh{X} \triangleq \left( \begin{array}{l}
X \\
1
\end{array} \right)$. {Extensions to multi-action space and nonlinear decision rules are possible but will necessitate advanced modeling and treatment.  This will be
left for future research.}
Using functional margin representation in standard classification, we then have $\II\left( A = d( X ) \right) = \II\left( A \, f(X;\theta) > 0 \right)$
for any nonzero $f(X;\theta)$.
Therefore, 
the  IDR-CDE optimiation problem can be equivalently written as:
\begin{equation} \label{eq:complete EV-formulation}
\displaystyle{
	\operatornamewithlimits{\mbox{sup}}_{\theta \triangleq ( \beta,\beta_0 ) \in \mathbb{R}^{p+1}, \,
	\alpha \in \calF}
} \, \left\{ \begin{array}{l}
\E\left[ \, {\cal Z} \, \displaystyle{
	\frac{\II( A \, f(X;\theta) > 0 )}{\pi(A | X)}
} \, \right] + \\ [0.2in]
\E\left[ \, \left[ \, \alpha(X) - {\cal Z} + u( {\cal Z} - \alpha(X) ) \, \right] \, \displaystyle{
	\frac{\II( A \, f(X;\theta) > 0 )}{\pi(A | X)}
} \, \right]
\end{array} \right\}.
\end{equation}
Before proceeding, we describe two characteristics of this problem that are important in the algorithmic development
and provide our proposal to address them. 

\gap
\noindent
{\bf (a) The discontinuity of the indicator function.} The function
$\II( A \, f(X;\theta) > 0 )$ is a lower semicontinuous, albeit discontinuous function.
This seems to prohibit us from employing continuous optimization algorithms to solve problem \eqref{eq:complete EV-formulation}.
A natural way to resolve this issue is to approximate the indicator function by a continuous function,
such as the piecewise truncated hinge loss as in \cite{wu2007robust}:
\[
T_{\delta}(x) \, \triangleq \, \displaystyle{
	\frac{1}{2 \, \delta}
} \, \underbrace{\left[ \, \max\left( \, x + \delta, 0 \, \right) - \max\left( \, x - \delta, 0 \, \right) \, \right]}_{\mbox{nonnegative}}\epc \mbox{for some $\delta > 0$},
\]
so that
\[ \begin{array}{lll}
\II\left( A \, f(X;\theta ) \, > \, 0 \right) & \approx & T_{\delta}(A \, f(X;\theta)) \\ [5pt]
& = & \underbrace{\displaystyle{
		\frac{1}{2 \, \delta}
	} \, \max\left( \, A \, f(X;\theta) + \delta, 0 \, \right)}_{\mbox{denoted $T_{\delta}^+(\theta;X,A)$}} -
\underbrace{\displaystyle{
		\frac{1}{2 \, \delta}
	} \, \max\left( \, A \, f(X;\theta) - \delta, 0 \, \right)}_{\mbox{denoted $T_{\delta}^-(\theta;X,A)$}},
\end{array} \]
where both functions $T_{\delta}^{\pm}(\bullet;X,A)$ are nonnegative, convex, and piecewise affine; thus
the approximating function is non-convex and non-differentiable, making the resulting optimization problem:
\begin{equation} \label{eq:truncated hinge loss approx}
\begin{array}{l}
\displaystyle\operatornamewithlimits{sup}_{ \substack{\theta \triangleq ( \beta,\beta_0 ) \in \mathbb{R}^{p+1}, \\[0.04in]
	\alpha \in \calF}}
\left\{
\begin{array}{ll}
\E\left[ \, {\cal Z} \, \displaystyle{
	\frac{T_{\delta}^+(\theta;X,A) - T_{\delta}^-(\theta;X,A)}{\pi(A | X)}
} \, \right] + \\ [0.2in]
\; \E\left[ \, \left[ \, \alpha(X) - {\cal Z} + u( {\cal Z} - \alpha(X) ) \, \right] \, \displaystyle{
	\frac{T_{\delta}^+(\theta;X,A) - T_{\delta}^-(\theta;X,A)}{\pi(A | X)}
} \, \right]\end{array}\right\}
\end{array}
\end{equation}
difficult to solve.  Since we are interested in designing an algorithm that is provably convergent to a properly
defined stationary solution, care is needed to handle the combined features of non-convexity and
non-differentiability in the approximated problem (\ref{eq:truncated hinge loss approx}) and the discontinuity in
\eqref{eq:complete EV-formulation}.  These features are particularly relevant when we consider
the convergence of the former to the latter as
$\delta \downarrow 0$.  To illustrate the difficulty with some algorithms for solving (\ref{eq:truncated hinge loss approx}),
we mention that a majorization-minimization type algorithm \cite{Lange2016MM} may be too complex to implement as
a majorizing function may be quite complicated; block coordinate descent type methods may not converge to a stationary point
of this problem because the needed regularity assumptions \cite{tseng2001convergence} cannot be expected to be satisfied.
Therefore, an alternative way to
tackle the discontinuity of the indicator function is needed, which is the focus of Subsection~\ref{subsec:reformulation}.

\gap
\noindent
{\bf (b) The positive scale-invariance of the indicator function.}  The function $\II( A \, f(X;\theta) > 0 )$ is
positively scale-invariant as any positive scaling of $f(X;\theta)$ will not change the objective value of the
problem \eqref{eq:complete EV-formulation}.  This could cause computational instability, and more seriously, incorrect definition
of the indicator function due to round-off errors; these numerical issues become more pronounced when $f(X;\theta)$ is close to 0
in practical implementation of an algorithm.
One way to guard against such undesirable characteristics of the indicator function is to solve two optimization problems with the
bias term
$\beta_0$ set equal to $\pm 1$, respectively, and accept as the solution the one with a smaller objective value.  In the development
below, this safe guard is adopted as can be seen in the formulation (\ref{eq:OCE binaryoptimization problem}).


\subsection{Difference-of-convex reformulation of (\ref{eq:complete EV-formulation})}
\label{subsec:reformulation}
In this subsection, we propose a method to transform the discontinuous optimization problem \eqref{eq:complete EV-formulation}
that involves the indicator function to a continuous optimization problem by means of a mild assumption.
Our approach is to reformulate the discontinuous problem \eqref{eq:complete EV-formulation} via its epigraphical representation.
Since $\II( \, \bullet\, > 0 )$ is a lower semicontinuous function, its epigraph
\[
{\rm epi}\,\II(  \, \bullet\, > 0 ) \,\triangleq \, \left\{\,(t, s)\in\mathbb{R}\times \mathbb{R}\mid t\geq \II( s > 0 )\,\right\}
\]
is a closed set \cite[Theorem 7.1]{rockafellar1970convex}.
However, the random variable $\mathcal{Z}$ may attain positive values, which makes it also essential to consider the
hypograph of $\II(\bullet > 0 )$, i.e., the set
\[
{\rm hypo}\,\II(  \, \bullet\, > 0 ) \,\triangleq \, \left\{\,(t, s)\in\mathbb{R}\times \mathbb{R}\mid t\leq \II( s > 0 )\,\right\}.
\]
Since the indicator function is not upper semicontinuous, the above set is not closed. We thus consider an approximation
of $\II(\bullet > 0 )$ by an upper semicontinuous function $\II(\bullet \geq 0 )$ that has a closed hypograph
\[
{\rm hypo}\,\II( \, \bullet\, \geq 0 ) \,\triangleq \, \left\{\,(t, s)\in\mathbb{R}\times \mathbb{R}\mid t\leq \II( s \geq 0 )\,\right\}.
\]
Interestingly, the sets ${\rm epi}\,\II(  \, \bullet\,  > 0 )$ and ${\rm hypo}\,\II(  \, \bullet\,\geq 0 )$ are each a finite union of polyhedra
that admits an extremely simple dc representation given in the next lemma. See also Figures~\ref{fig:epi} and \ref{fig:hypo} for illustration.
No proof is required for the lemma.

\begin{lemma}\label{lemma:epi representation}
For any $t,s\in \mathbb{R}$, the following two statements hold:

(i) $(t,s)\in {\rm epi}\,\II( \bullet > 0 )$ if and only if 
$\max(-t, s) - \max(t+s-1, 0)\leq 0$\,;

(ii) $(t,s)\in {\rm hypo}\,\II( \bullet \geq 0 )$ if and only if 
$\max(t+s-1, 0) - \max(-t, s) \leq 0$. \hfill $\Box$
\end{lemma}

\begin{figure}[h]
	\centering
	\fbox{
		\begin{minipage}{.4\textwidth}
			\begin{center}
				\begin{tikzpicture}[scale = 0.8]
				
				\draw[gray!40, thin, step=0.5] (-2.5,-2) grid (2.5,2.5);
				\draw[thick,->] (-2.7,0) -- (2.7,0) node[right] {$s$};
				\draw[thick,->] (0,-2.2) -- (0,2.8) node[above] {$t$};
				
				\foreach \x in {-2,...,-1} \draw (\x,0.05) -- (\x,-0.05) node[below] {\tiny\x};
				\foreach \x in {1,...,2} \draw (\x,0.05) -- (\x,-0.05) node[below] {\tiny\x};
				\foreach \y in {-2,...,2} \draw (-0.05,\y) -- (0.05,\y) node[right] {\tiny\y};
				
				\draw[red, line width=0.5mm] (-2.5,0) -- (0,0);
				\draw[red, line width=0.5mm] (0.07,1) -- (2.5,1);
				\draw [red,thick, fill = white] (0,1) circle (0.7mm);
				\draw [red,fill=red] (0,0) circle (0.7mm);
				
				\fill[opacity=0.6,pattern=north east lines] (0,1) -- (2.5,1) -- (2.5,2.5) -- (0,2.5) -- cycle;
				\fill[opacity=0.6,pattern=north east lines] (-2.5,0) -- (0,0) -- (0,2.5) -- (-2.5,2.5) -- cycle;
				\end{tikzpicture}
				\caption{\small the region (shaded) for ${\rm epi}\,\II( \bullet > 0 )$}
				\label{fig:epi}
			\end{center}
		\end{minipage}
		\quad
		\begin{minipage}{.4\textwidth}
			\begin{center}
				\begin{tikzpicture}[scale = 0.8]
				
				\draw[gray!40, thin, step=0.5] (-2.5,-2) grid (2.5,2.5);
				\draw[thick] (-2.7,0) -- (-0.07,0);
				\draw[thick,->] (0.07,0) -- (2.7,0) node[right] {$s$};
				\draw[thick] (0,-2.2) -- (0,-0.07);
				\draw[thick,->] (0,0.07) -- (0,2.8) node[above] {$t$};
				
				\foreach \x in {-2,...,-1} \draw (\x,0.05) -- (\x,-0.05) node[below] {\tiny\x};
				\foreach \x in {1,...,2} \draw (\x,0.05) -- (\x,-0.05) node[below] {\tiny\x};
				\foreach \y in {-2,...,2} \draw (-0.05,\y) -- (0.05,\y) node[right] {\tiny\y};
				
				\draw[red, line width=0.5mm] (-2.5,0) -- (-0.07,0);
				\draw[red, line width=0.5mm] (0,1) -- (2.5,1);
				\draw [red,fill=red] (0,1) circle (0.7mm);
				\draw [red,thick,fill = white] (0,0) circle (0.7mm);
				
				\fill[pattern=north east lines,opacity=0.6] (0,1) -- (2.5,1) -- (2.5,-2) -- (0,-2) -- cycle;
				\fill[pattern=north east lines,opacity=0.6] (-2.5,0) -- (0,0) -- (0,-2) -- (-2.5,-2) -- cycle;
				\end{tikzpicture}
				\caption{\small the region (shaded) for ${\rm hypo}\,\II( \bullet \geq 0 )$}
				\label{fig:hypo}
			\end{center}
		\end{minipage}
	}
\end{figure}

Denoting $\calZ^- \triangleq \max(-\calZ, 0)$ and $\calZ^+ \triangleq \max(\calZ, 0)$, we
assume that
\[
\E\left[ \, \calZ^+ \, \displaystyle{
\frac{\II( A \, f(X;\theta) = 0 )}{\pi(A | X)}
} \, \right] = 0.
\]
Under this assumption, problem \eqref{eq:complete EV-formulation} is equivalent to
\begin{equation} \label{eq:complete EV-formulation 2}
\displaystyle{
	\operatornamewithlimits{\mbox{minimize}}_{ \beta \in \mathbb{R}^{p}, \alpha \in \calF}
} \, \left\{ \begin{array}{l}
\E\left[ \, \calZ^- \, \displaystyle{
	\frac{\II( A \, f(X;\theta) > 0 )}{\pi(A | X)}
} \, \right] - \E\left[ \, \calZ^+ \, \displaystyle{
	\frac{\II( A \, f(X;\theta) \geq 0 )}{\pi(A | X)}
} \, \right]\\ [0.2in]
\displaystyle{
} \, +\E\left[ \, \left( \, \underbrace{{\cal Z} - \alpha(X) - u( {\cal Z} - \alpha(X))}_{\mbox{nonnegative}} \, \right) \, \displaystyle{
	\frac{\II(A \, f(X;\theta) > 0 )}{\pi(A | X)}
} \, \right]
\end{array} \right\}.
\end{equation}
For further consideration, we take $\alpha(X)$ to be a parameterized family of affine functions
$\{ b^{\, T} X + \beta_0 = w^{\, T} \wh{X} \}$ where $w \triangleq \left( \begin{array}{c}
b \\
b_0
\end{array} \right)$ is the parameter is be estimated. {The use of affine functions to approximate $\alpha^\ast(X)$ is based on both modeling and computational perspectives. The affine functions are easy for interpretation, but may suffer from model misspecification. The linear assumption can be relaxed by using kernel trick in machine learning. The corresponding computation will be more involved.}
%
We approximate the expectation in \eqref{eq:complete EV-formulation 2} by the sample average that  is
based on the available data $\{ ( X^{\,i},A_i,{\cal Z}_{\, i} ) \}_{i=1}^N$. In order to compute a sparse solution
that can avoid model overfitting, we add sparsity surrogate functions \cite{ahn2017difference} $P_b$ and $P_{\theta}$
on the parameters $w$ and $\beta$ in the
covariate function $\alpha(X)$ and the function $f(X;\theta)$, respectively,
each weighted by the positive scalars $\lambda_b^N$ and $\lambda_{\beta}^N$. The empirical problem is then given by
{\small \begin{equation} \label{eq:OCE binaryoptimization problem}
\begin{array}{l}
\displaystyle\operatornamewithlimits{minimize}_{\substack{\beta \in \mathbb{R}^{p}\\ w \triangleq ( b,b_0 ) \in S}}
\left\{ \begin{array}{l}
\lambda_a^N \, P_b(b) + \lambda_{\beta}^N \, P_{\beta}(\beta) + \displaystyle{
	\frac{1}{N}
} \, \displaystyle{
	\sum_{i=1}^N
} \, {\cal Z}^-_i \,  \displaystyle{
	\frac{\II( A_{\,i} \, (\beta^{\, T} X^{\,i} \pm 1)  >0 )}{\pi(A_{\,i} \,|\, X^{\,i})}
} -\\[0.2in]

 \,  \displaystyle{
	\frac{1}{|\calN_+|}
} \,\displaystyle{
	\sum_{i \in \calN_+}
} \, {\cal Z}^+_i \,  \displaystyle{
	\frac{\II( A_{\,i} \, (\beta^{\, T} X^{\,i} \pm 1)  \geq 0 )}{\pi(A_{\,i} \,|\, X^{\,i})}
} +\\ [0.2in]

\displaystyle{
	 \frac{1}{N}
} \, \displaystyle{
	\sum_{i=1}^N
} \, \left[ \, {\cal Z}_i - w^{\, T} \wh{X}^{\, i} - u( {\cal Z}_i - w^{\, T} \wh{X}^{\, i} ) \, \right] \, \displaystyle{
	\frac{\II( A_{\,i} \, (\beta^{\, T} X^{\,i} \pm 1)  > 0 )}{\pi(A_{\,i} \,|\, X^{\,i})}
}
\end{array} \right\},
\end{array}
\end{equation}}
where $\calN_+ \,\triangleq\, \{\,1\leq j\leq N \ | \ \calZ_j > 0 \,\}$ and $S$ is a closed convex set.
[In principle, we may add constraints to the parameter $\beta$ also but refrain from doing this as it does
not add value to the methodology.]
Based on Lemma \ref{lemma:epi representation}, the above problem can be further written as
{\small \begin{equation} \label{eq:OCE optimization problem 3}
\begin{array}{l}
\mbox{minimize over $z \, \triangleq \, ( w,\beta,\sigma^\pm )$; \, $\beta \, \in \, \mathbb{R}^{p}$,
and $w \, \triangleq \, ( b,b_0 ) \, \in \, S$} \\ [0.1in]
\varphi(z) \, \triangleq \, \left\{ \begin{array}{l}
\lambda_a^N \, P_b(b) + \lambda_{\beta}^N \, P_{\beta}(\beta) + \displaystyle{
	\frac{1}{N}
} \, \displaystyle{
	\sum_{i=1}^N
} \,   \displaystyle{
	\frac{{\cal Z}^-_i \sigma^-_i}{\pi(A_{\,i} \,|\, X^{\,i})}
} -  \displaystyle{
	\frac{1}{|\calN_+|}
} \displaystyle{
	\sum_{j \in \calN_+}
} \, \displaystyle{
	\frac{{\cal Z}^+_j \sigma^+_j}{\pi(A_{\,j} \,|\, X^{\,j})}
} \\ [0.2in]
\displaystyle{
	\frac{1}{N}
} \, \displaystyle{
	\sum_{i=1}^N
} \, \underbrace{\left[ \, {\cal Z}_i - w^{\, T} \wh{X}^{\, i} - u( {\cal Z}_i - w^{\, T} \wh{X}^{\, i} ) \, \right] \,
\displaystyle{
\frac{\sigma_i^-}{\pi(A_{\,i} \,|\, X^{\,i})}
}}_{\mbox{nonconvex}}
\end{array} \right\}\\ [0.5in]
\mbox{subject to} \\ [5pt]
\max(-\sigma_i^-\, , \,A_{\,i} \, (\beta^{\, T} X^{\,i} \pm 1)) - \max(\sigma_i^- + \, A_{\,i}  (\beta^{\, T} X^{\,i} \pm 1) - 1,\, 0) \, \leq \, 0, \;
1 \leq i \leq N \\ [0.1in]
\max(\sigma_{\,j} ^+ +  A_{\,j}  \, (\beta^{\, T} X^j \pm 1)-1) - \max(-\sigma_{\,j} ^+\,,\, A_{\,j}  \, (\beta^{\, T} X^j \pm 1)) \leq 0, \epc j \in \calN_+,
\end{array}
\end{equation}
}
where the constraints are of the difference-of-convex, piecewise affine type.
Denote $t_{\,i} \triangleq {\cal Z}_i - w^{\, T} \wh{X}^{\, i}$ for any $i = 1,\cdots, N$.
The last term in the objective function $\varphi$ can be further written as
\[ \begin{array}{ll}
& \left[ \, t_{\,i} - u(t_{\,i}) \, \right] \, \displaystyle{
	\frac{\sigma_i^-}{\pi(A_{\,i} \,|\, X^{\,i})}
} \\[0.1in]
 =  &  \displaystyle{
	\frac{1}{2 \, \pi(A_{\,i} \,|\, X^{\,i})}
} \,
\left\{ \, \left[ \, t_{\,i} - u(t_{\,i}) + \sigma_i^- \, \right]^2 - (\sigma_i^{-})^2 - [\,t_{\,i} - u(t_{\,i})\,]^2 \, \right\}.
\end{array} \]
Since $t_{\,i} - u(t_{\,i})\geq 0$ and $\sigma_{i}^-\geq 0$, the terms $\left[ \, t_{\,i} - u(t_{\,i}) + \sigma_i^- \, \right]^2$
and $[\,t_{\,i} - u(t_{\,i})\,]^2$ are convex.  Hence each product $\left[ \, t_{\,i} - u(t_{\,i}) \, \right] \, \displaystyle{
	\frac{\sigma_i^-}{\pi(A_{\,i} \,|\, X^{\,i})}
}$ is the difference of convex functions.

\gap

Suppose that the utility function and sparsity surrogate functions are as follows:
\begin{equation}\label{eq:utility and sparsity}
\begin{array}{rll}
u(t) &  = & \xi_1 \, \max( 0,t ) - \xi_2 \, \max( 0,-t ), \epc  \mbox{where}\; 0 \, \leq \, \xi_1 \, < \, 1 \, < \, \xi_2\,; \\[0.1in]
P_b(b)& = & \displaystyle{
	\sum_{i=1}^p
} \, \left[ \, \phi_i^b | \, b_{\,i} \, | - \rho_i^b(b_i) \, \right] , \epc \phi_i^b \, > \, 0, \ i \, = \, 1, \cdots, p\,; \\ [0.2in]
P_\beta(\beta) & = & \displaystyle{
	\sum_{i=1}^p
} \, \left[ \, \phi_i^\beta | \, \beta_{\,i} \, | - \rho_i^\beta(\beta_i) \, \right] , \epc \phi_i^\beta \, > \, 0, \ i \, = \, 1, \cdots, p,
\end{array}
\end{equation}
where $\phi_i^b$ and $\phi_i^{\beta}$ are given constants and $\rho_i^b$ and $\rho_i^{\beta}$ are convex differentiable
functions \cite{ahn2017difference}. We then have
\[
\begin{aligned}
\left[ \, t_{\,i} - u(t_{\,i}) \, \right] \, \displaystyle{
	\sigma_i^-
}
\, =\,  \displaystyle \frac{1}{2} \, \bigg\{\,\underbrace{(1-\xi_1)\left[ \, \max(0,t_i) + \sigma_i^- \, \right]^2 + (1+\xi_2)\left[ \, \max(0,-t_i) + \sigma_i^- \, \right]^2}_{\mbox{convex}} \\[0.05in]
 \ - \underbrace{\left[\,(2-\xi_1+\xi_2)(\sigma_i^-)^2 - (1-\xi_1)\left[\,\max(0,t_i)\,\right]^2 - (1+\xi_2)\left[\,\max(0,-t_i)\,\right]^2\,\right]}_{\mbox{convex and continuously differentiable}}\,\bigg\}.
\end{aligned}
\]
Therefore, under the above setting,
the objective function $\varphi$ is the difference of two convex functions, $\varphi_1 - \varphi_2$, with $\varphi_2$ being continuously differentiable.
In the next section, we present a dc algorithm for solving such a problem.


\section{Solving a Piecewise Affine Constrained DC Program}\label{sec:dca}

We consider problem \eqref{eq:OCE optimization problem 3} cast in the following general form:
\begin{equation}\label{eq: dc1}
\begin{array}{ll}
\displaystyle\operatornamewithlimits{minimize}_{x\in X} \epc  f(x) \, - \, g(x)\\[0.15in]
\mbox{subject to} \\[0.1in]
\epc \displaystyle\max_{1\, \leq \, j \, \leq \, J_{1i}}((a^{\,ij})^Tx + \alpha_{\,ij}) \, - \,
 \max_{1 \, \leq \,j \, \leq \,J_{2i}}((b^{\,ij})^Tx + \beta_{\,ij})\, \leq \, 0, \epc i = 1, \ldots, m,
\end{array}
\end{equation}
where $f:\mathbb{R}^n \to \mathbb{R}$ is a convex function, $g:\mathbb{R}^n\to \mathbb{R}$ is a continuously differentiable convex function with
Lipschitz continuous gradient, each $a^{\,ij}$ and $b^{\,ij}$ are  $n$-dimensional vectors, each $\alpha_{\,ij}$ and $\beta_{\,ij}$ are scalars,
each $J_{1i}$ and $J_{2i}$ are positive integers, and $X$ is a polyhedral set.
Notice that for any $i = 1, \ldots, m$, it holds that
\[\begin{array}{ll}
& \displaystyle\max_{1 \, \leq \, j \, \leq \, J_{1i}}((a^{\,ij})^Tx + \alpha_{\,ij}) \, - \, \max_{1\, \leq \,j \, \leq \,
J_{2i}}((b^{\,ij})^Tx + \beta_{\,ij})\, \leq \, 0 \\[0.2in]
\Longleftrightarrow &  (a^{\,ij_1})^Tx + \alpha_{\,ij_1} - \displaystyle\max_{1\leq j \leq J_{2i}}((b^{\,ij})^Tx + \beta_{\,ij})\, \leq \, 0, \epc
\forall\; 1\leq j_1 \leq J_{1i} \\[0.2in]
\Longleftrightarrow &   \displaystyle\max_{1\leq j_2 \leq J_{2i}}\left(\,(b^{\,ij_2} - a^{\,ij_1})^Tx + (\beta_{\,ij_2}- \alpha_{\,ij_1})\,\right)\, \geq \, 0,
\epc \forall\; 1\leq j_1 \leq J_{1i}.
\end{array}
\]
The above equivalences indicate that by properly redefining $(b^{\,ij}, \beta_{\,ij})$ and the value of $m$,
one can write any piecewise linear constrained dc program \eqref{eq: dc1}  as the following reverse convex constrained \cite{hillestad1980reverse} dc program:
\begin{equation}\label{eq: dc2}\begin{array}{ll}
\displaystyle\operatornamewithlimits{minimize}_{x\in X} \epc & h(x)\,\triangleq \, f(x) \, - \, g(x)\\[0.15in]
\mbox{subject to} \epc &   \, \displaystyle\max_{1\, \leq \, j \, \leq \,J_{\,i}}(\,(\,b^{\,ij}\,)^Tx + \beta_{\,ij}\,)\, \geq \, 0, \epc  i = 1, \ldots, m.
\end{array}
\end{equation}
Denote the feasible set of the problem \eqref{eq: dc2} as
\[
F\,\triangleq \,\left\{x\in X\mid \displaystyle{
\max_{1\, \leq \, j \, \leq \,J_{\,i}}
} \, (\,(\,b^{\,ij}\,)^Tx + \beta_{\,ij}\,) \, \geq \, 0, \epc  i = 1, \ldots, m\,\right\}.
\]
For any $x\in \mathbb{R}^n$, we also denote
\[
{\cal I}(x) \,\triangleq\, \left\{\,1\leq i\leq m\mid \, \displaystyle\max_{1\leq j \leq J_{\,i}}(\,(\,b^{\,ij}\,)^Tx + \beta_{\,ij}\,)\, =\, 0\,\right\}
\]
and
\[
\mathcal{A}_{\,i}(x) \,\triangleq\, \displaystyle{
\operatornamewithlimits{argmax}_{1\, \leq \, j \, \leq \, J_{\,i}}
} \, \left\{\,(\,b^{\,ij}\,)^Tx + \beta_{\,ij}\,\right\}, \epc i = 1, \ldots, m.
\]
We say that $\bar{x}\in X$ is a B(ouligand)-stationary point \cite{pang2007partially} of the problem \eqref{eq: dc2} if
\[
h^{\prime}(\bar{x};d)\,\triangleq\, \operatornamewithlimits{lim}_{\tau\downarrow 0 }
\, \displaystyle{
	\frac{h(\bar{x} + \tau d) - h(\bar{x})}{\tau}
}  = f^{\prime}(\bar{x};d) - g^{\prime}(\bar{x};d)\,\geq\, 0, \epc \forall\; d\in \mathcal{T}_{\,B\,}(\bar{x};F),
\]
where $\mathcal{T}_{\,B\,}(\bar{x};F)$ is the Bouligand tangent cone of $F$ at $\bar{x}\in F$, i.e., (see, e.g.,~\cite[Proposition~3]{pang2016computing}),
\[\begin{array}{rl}
\mathcal{T}_{\,B\,}(\bar{x};F) \, \triangleq & \left\{\,d\in \mathbb{R}^n\, \mid \,
d \, = \, \displaystyle{
\lim_{\nu\to\infty}
} \, \displaystyle{
\frac{(x^{\,\nu} - \bar{x} )}{\tau_{\nu}},
} \ \mbox{where $F\ni x^{\,\nu}\to \bar{x}$ and $\tau_{\nu}\downarrow 0$}  \right\}\\[0.2in]
= & \left\{\,d\in \mathcal{T}_{\,B\,}(\bar{x};X) \, \mid \,  \displaystyle\max_{j \in \mathcal{A}_{\,i}(\bar{x})}\, (b^{\,ij})^{\,T}d\geq \,0, \;\forall\;
i \in {\cal I}(\bar{x})\,   \right\}  \\ [0.2in]
= & \displaystyle{
\bigcap_{i \in {\cal I}(\bar{x})}
} \, \displaystyle{
\bigcup_{j \in {\cal A}_i(\bar{x})}
} \, \left\{\,d\in \mathcal{T}_{\,B\,}(\bar{x};X) \, \mid \, (b^{\,ij})^{\,T}d \, \geq \,0 \, \right\}.
\end{array}
\]
[Since $X$ is assumed to be polyhedral, $\mathcal{T}_{\,B\,}(\bar{x};X)$ is a polyhedral cone.]
A weaker concept than B-stationarity is that of weak B-stationarity, which pertains to a feasible solution $\bar{x}\in F$ such that $h^{\prime}(\bar{x};d)\geq 0$ for any $d\in \mathbb{R}^n$ satisfying
\[ \begin{array}{rl}
d\in \mathcal{T}^{\,\rm weak}_{\,B\,}(\bar{x};F) \,\triangleq &
\left\{ \, d\in \mathcal{T}_{\,B\,}(\bar{x};X) \, \mid \,  \displaystyle\min_{j \in \mathcal{A}_{\,i}(\bar{x})}\, (b^{\,ij})^{\,T}d\geq \,0,
\;\forall\; i \in {\cal I}(\bar{x})\,   \right\} \\ [0.2in]
= & \displaystyle{
\bigcap_{i \in {\cal I}(\bar{x})}
} \, \displaystyle{
\bigcap_{j \in {\cal A}_i(\bar{x})}
} \, \left\{\,d\in \mathcal{T}_{\,B\,}(\bar{x};X) \, \mid \, (b^{\,ij})^{\,T}d \, \geq \,0 \, \right\}.
\end{array}
\]
Unlike $\mathcal{T}_{\,B\,}(\bar{x};F)$, which is not necessarily convex, $\mathcal{T}_{\,B\,}^{\, \rm weak}(\bar{x};F)$ is a polyhedral cone.
It is known from \cite[Chapter 2, Proposition 1.1(c) \& Exercise 9.10]{clarke1998nonsmooth} that
\[
\mathcal{T}_{C}(\bar{x}; F)\,\subseteq\, \mathcal{T}_{B}^{\,\rm weak}(\bar{x};F)\,\subseteq \,\mathcal{T}_{B}(\bar{x};F),
\]
where $\mathcal{T}_{\,C\,}(\bar{x}; F)$ denotes the Clarke tangent cone of $F\subseteq\mathbb{R}^n$ at $\bar{x}$,
i.e., $d\in \mathcal{T}_{\,C\,}(\bar{x}; F)$ if for every sequence $\{x^{\,i}\}\subseteq S$ converging to $\bar{x}$ and positive
scalar sequence $\{t_{\,i}\}$ decreasing to $0$, there exists a sequence $\{d^{\,i}\}\subseteq \mathbb{R}^n$ converging to $d$ such that
$x^{\,i} + t_{\,i}\,d^{\,i} \in F$ for all $i$ \cite[Chapter 2, Proposition 5.2]{clarke1998nonsmooth}.

In order to better understand the above two stationarity concepts in the context of the piecewise polyhedral structure of the feasible set $F$ and to
motivate the algorithm to be presented afterward for solving the problem \eqref{eq: dc2}, we first introduce a further stationarity concept, which we call
A-stationarity (A for Algorithm).  Specifically, we note that $F$ is the union of finitely many polyhedra:
\[
F \, = \, \displaystyle{
\bigcup_{(j_1, \cdots, j_m)}
} \, \left\{ \, x \, \in \, X \, \mid \, (\, b^{\,ij_i} \,)^Tx + \beta_{ij_i} \, \geq \, 0, \epc  i = 1, \ldots, m\,\right\},
\]
where the union ranges over all tuples $\{ j_i \}_{i=1}^m$ with each $j_i \in \{ 1, \cdots, J_i \}$ for all $i$.  Given a vector $\bar{x} \in F$,
let ${\cal J}(\bar{x})$ be the family of such tuples such that $j_i \in {\cal A}_i(\bar{x})$ for all $i = 1, \cdots, m$.  We say that $\bar{x} \in F$
is {\sl A-stationary} if there exists a tuple $\bar{j}(\bar{x}) = \{ \, \bar{j}_i \, \}_{i=1}^m \in {\cal J}(\bar{x})$ such that
\[
h^{\, \prime}(\bar{x};d) \, \geq \, 0, \ \forall \, d \, \in {\cal T}_A^{\bar{j}(\bar{x})}(\bar{x};F) \, \triangleq \,
\left\{ \, d \, \in \, {\cal T}_{\,B\,}(\bar{x};X)
\, \mid \, ( \, b^{i \bar{j}_i} \, )^T d \, \geq \, 0, \ \forall \, i \, \in \, {\cal I}(\bar{x}) \, \right\}.
\]

\begin{lemma}\label{lemma:weak B}
Let $\bar{x}\in F$ be given.  Consider the following statements all pertaining to the problem \eqref{eq: dc2}:\\[0.05in]
(a) $\bar{x}$ is B-stationary;\\[0.05in]
(b) $\bar{x}$ is A-stationary;\\[0.05in]
(c) there exists a tuple $\bar{j}(\bar{x}) = \{ \, \bar{j}_i \, \}_{i=1}^m \in {\cal J}(\bar{x})$ such that
\begin{equation} \label{eq:A-stationary in terms of min}
\bar{x} \, \in \, \displaystyle{
\operatornamewithlimits{\mbox{argmin}}_{x\in X}
} \, \left\{ \, f(x) - [\,g(\bar{x}) + \nabla g(\bar{x})^T(x-\bar{x})\,]\,\mid \, (b^{\,i\,\bar{j}_{\,i}})^Tx +
\beta_{\,i\,\bar{j}_{\,i}}\geq 0, \; i \in {\cal I}(\bar{x}) \, \right\};
\end{equation}
(d) there exists a tuple $\bar{j}(\bar{x}) = \{ \, \bar{j}_i \, \}_{i=1}^m \in {\cal J}(\bar{x})$ such that
\[
\bar{x}  \in \, \displaystyle{
\operatornamewithlimits{\mbox{argmin}}_{x\in X}
}  \left\{ \, f(x) - [\,g(\bar{x}) + \nabla g(\bar{x})^T(x-\bar{x})\,]\,\mid \, (b^{\,i\,\bar{j}_{\,i}})^Tx +
\beta_{\,i\,\bar{j}_{\,i}}\geq 0, \; i = 1, \cdots, m \, \right\};
\]
(e) $\bar{x}$ is weak B-stationary.\\
It holds that (a) $\Rightarrow$ (b) $\Leftrightarrow$ (c) $\Leftrightarrow$ (d) $\Rightarrow$ (e).
\end{lemma}
\begin{proof}
(a) $\Rightarrow$ (b).  This is because ${\cal T}_A^{\bar{j}(\bar{x})}(\bar{x};F) \subseteq {\cal T}_B(\bar{x};F).$

\gap

(b) $\Rightarrow$ (e).  This is because $\mathcal{T}^{\,\rm weak}_{\,B\,}(\bar{x};F) \subseteq {\cal T}_A^{\bar{j}(\bar{x})}(\bar{x};F)$.

\gap

(b) $\Leftrightarrow$ (c).  This is clear because the condition  $h^{\, \prime}(\bar{x};d) \geq 0$ for all
$d \in {\cal T}_A^{\bar{j}(\bar{x})}(\bar{x};F)$ is
exactly the first-order optimality condition of the convex program in (\ref{eq:A-stationary in terms of min}).

\gap

(c) $\Rightarrow$ (d).  This is clear because there are more constraints in the feasible region of the optimization problem in (d) than those
in (c).

\gap

(d) $\Rightarrow$ (c).   Let $x \in X$ satisfy $(b^{\,i\,\bar{j}_{\,i}})^Tx +
\beta_{\,i\,\bar{j}_{\,i}}\geq 0$ for all $i \in {\cal I}(\bar{x})$.  Since
$(b^{\,i\,\bar{j}_{\,i}})^Tx +
\beta_{\,i\,\bar{j}_{\,i}} > 0$ for all $i \not\in {\cal I}(\bar{x})$, it follows that for all $\tau > 0$ sufficiently small,
the vector $x^{\tau} \triangleq x + \tau ( \bar{x} - x )$ satisfies
$(b^{\,i\,\bar{j}_{\,i}})^Tx^{\tau} +
\beta_{\,i\,\bar{j}_{\,i}} \geq 0$ for all $i = 1, \cdots, m$.  Hence,
\[ \begin{array}{lll}
f(\bar{x}) - g(\bar{x}) & \leq & f(x^{\tau}) - \left[\,g(\bar{x}) + \nabla g(\bar{x})^T(x^{\tau} - \bar{x})\, \right] \epc \mbox{by (d)} \\ [0.1in]
& \leq & \tau \, \left[ \, f(\bar{x}) - g(\bar{x}) \, \right] +
( \, 1 - \tau \, ) \, \left[ \, f(x) - \left[ \,g(\bar{x}) + \nabla g(\bar{x})^T( \, x - \bar{x} \, ) \, \right] \,\right],
\end{array} \]
which yields
\[
f(\bar{x}) - g(\bar{x}) \, \leq \, f(x) - \left[ \,g(\bar{x}) + \nabla g(\bar{x})^T( \, x - \bar{x} \, ) \, \right],
\]
establishing (c).
\end{proof}

\gap

In the following, we propose a dc algorithm to compute an A-stationary point of \eqref{eq: dc2}. The algorithm takes
advantage of the reverse convex constraints of the problem in that once initiated at a feasible vector $x^0 \in F$, the algorithm
generates a feasible sequence $\{ x^{\nu} \} \subset F$; see Step~1 below.

\noindent\makebox[\linewidth]{\rule{\textwidth}{1pt}}
\noindent
A dc algorithm for solving the  reverse convex constrained dc program \eqref{eq: dc2}.

\noindent\makebox[\linewidth]{\rule{\textwidth}{1pt}}

\noindent
{\bf Initialization.}  Given are a scalar $c > 0$, an initial point $x^{\, 0} \in F$.

\gap

{\bf Step 1.}  For each $i = 1, \cdots, m$, choose an index $j_{\,i}^{\,\nu}\in \mathcal{A}_{\,i\,}(x^{\,\nu})$.  Let
$x^{\,\nu+1}$ be the unique optimal solution of the convex program:
\begin{equation}\label{eq:dc sub}
\begin{array}{ll}
\displaystyle{
\operatornamewithlimits{\mbox{minimize}}_{x \in X}
} & \wh{h}_{\,c\,}(x;x^{\,\nu}) \, \triangleq  \,
f(x) - [ \, g(x^{\,\nu}) + (\nabla g(x^{\,\nu}))^{\,T}(x-x^{\,\nu})\,] \\[0.1in]
&\epc \epc \quad \quad \quad \quad +\underbrace{\displaystyle\frac{c}{2}\,\|x - x^{\,\nu}\|^2}_{\mbox{\small proximal regularization}} \\[0.1in]
\mbox{subject to} & (b^{\,i\,j_{\,i}^{\,\nu}})^{\,T} x+\beta_{\,i\,j_{\,i}^{\,\nu}} \, \geq \, 0,\epc i = 1,\ldots, m.
\end{array}
\end{equation}
{\bf Step 2.} If $x^{\, \nu+1}$ satisfies a prescribed stopping rule, terminate; otherwise, return to Step 1 with $\nu$ replaced by $\nu+1$.
\hfill $\Box$

\noindent\makebox[\linewidth]{\rule{\textwidth}{1pt}}

An enhanced version of the above algorithm that requires solving multiple subproblems for all indices $j_{\,i}^{\,\nu}$ in a so-called
``$\varepsilon$-argmax set'' has been suggested in \cite{pang2016computing}.  For this enhanced algorithm, it can be shown that every
accumulation point, if exists, of the generated sequence is a B-stationary point.  Although there are theoretical benefits of such an
algorithm, it may not be efficient when applied to the empirical CDE problem \eqref{eq:OCE optimization problem 3}, because the number
of reverse convex inequalities in the constraint set is proportional to the number of samples, making  the ``$\varepsilon$-argmax set''
potentially very large, thus potentially many subprograms need to be solved at every iteration.  There is also a probabilistic variant
of the enhanced algorithm that also solves only one convex subprogram of the same type as (\ref{eq:dc sub}).  The only difference from
the presented deterministic algorithm is that the tuple $\{ \, \bar{j}_i^{\nu} \, \}_{i=1}^m$ is chosen from the $\varepsilon$-argmax sets
randomly with positive probabilities.  Almost sure convergence of the probabilistic algorithm to a B-stationary point can be established.
Since the above (deterministic) algorithm has not been formally introduced in the literature, we provide below a (subsequential) convergence
result to an A-stationary solution of the problem \eqref{eq: dc2}.

It is worth mentioning that each $x^{\nu+1}$ is feasible to the subprogram (\ref{eq:dc sub}) at iteration $\nu+1$ because
\[
(b^{\,i\,j_{\,i}^{\,\nu+1}})^{\,T} x^{\nu+1} +\beta_{\,i\,j_{\,i}^{\,\nu+1}} \, = \,
\displaystyle{
\max_{1 \leq j \leq J_i}
} \, (\,(\,b^{\,ij}\,)^Tx^{\nu+1} + \beta_{\,ij}\,) \, \geq \, (b^{\,i\,j_{\,i}^{\,\nu}})^{\,T} x^{\nu+1} +\beta_{\,i\,j_{\,i}^{\,\nu}} \geq 0.
\]
This inequality also shows that $x^{\nu+1} \in F$ for all $\nu$.  The following theorem asserts the subsequential convergence of the
sequence generated by the above dc algorithm to an A-stationary point of problem \eqref{eq: dc2}.

\begin{theorem}\label{thm:subsequential convergence}
Suppose that $h$ is bounded below on the polyhedral set $X$. Then any accumulation point $x^{\,\infty}$ of the sequence $\left\{ x^{\, \nu} \right\}$
generated by the dc algorithm, if it exists, is an A-stationary point of \eqref{eq: dc2}.
\end{theorem}
\begin{proof}
	The  sequence of function values $\{h(x^{\,\nu})\}$ decreases since
	\[\begin{array}{rl}
	& h(x^{\,\nu+1}) + \displaystyle\frac{c}{2}\|x^{\,\nu+1} - x^{\,\nu}\|^2 \\[0.1in]
	 \leq & \wh{h}_{\,c\,}(x^{\,\nu+1}; x^{\,\nu}) \epc \mbox{(by the convexity of $g$)}\\[0.1in]
	\leq & h(x^{\,\nu}) \ \mbox{(by the optimality of $x^{\,\nu+1}$ and the feasibility of $x^{\,\nu}$  to \eqref{eq:dc sub})}.
	\end{array}
	\]
	Since $h$ is bounded below on $X$, we may derive that $\displaystyle\lim_{\nu\to \infty}\|x^{\,\nu+1} - x^{\,\nu}\| = 0$.
	By the definition of the point $x^{\,\nu+1}$, we obtain that for all $x\in X$ satisfying $(b^{\,i\,j_{\,i}^{\,\nu}})^Tx + \beta_{\,i\,j_{\,i}^{\,\nu}}\geq 0$, $i = 1, \ldots, m$,
	\begin{equation}\label{ineq:dca}
	\begin{array}{l}
	f(x^{\,\nu+1}) - \left[\,g(x^{\,\nu}) + \nabla g(x^{\,\nu})^T(x^{\,\nu+1}-x^{\,\nu})\,\right] + \displaystyle\frac{c}{2}\,\|x^{\,\nu+1} - x^{\,\nu}\|^2 \\[0.15in]
	\epc \leq \, f(x) - \left[\,g(x^{\,\nu}) + \nabla g(x^{\,\nu})^T(x-x^{\,\nu})\,\right] + \displaystyle\frac{c}{2}\,\|x - x^{\,\nu}\|^2.
	\end{array}
		\end{equation}
Let $\{x^{\,\nu+1}\}_{\nu\in \kappa}$ be a subsequence of $\{x^{\,\nu}\}$ that converges to $x^{\,\infty}$.
Then $x^{\,\infty}\in F$. Since each $\mathcal{A}_{\,i\,}(x^{\,\nu})$ is finite, we may assume without loss of generality that the
selected $j_{\,i}^{\,\nu}\in \mathcal{A}_{\,i\,}(x^{\,\nu})$ are independent of $\nu$ for any $i= 1,\ldots, m$ on this subsequence,
i.e., there exists $\bar{j}_{\,i}$ such that $\bar{j}_{\,i} = j_{\,i}^{\,\nu}$ for all
$i = 1, \ldots, m$ and all $\nu \in \kappa$.  For all $x\in X$ satisfying $(b^{\,i\,\bar{j}_i})^Tx + \beta_{\,i\,\bar{j}_i}\geq 0$,
the inequality \eqref{ineq:dca} holds.  Taking limit of $\nu(\in \kappa)\to +\infty$, we obtain that
$\bar{j}_{\,i}\in \mathcal{A}_{\,i\,}(x^{\,\infty})$ for $i = 1, \ldots, m$,  and for all $x\in X$ satisfying $(b^{\,i\,\bar{j}_i})^Tx + \beta_{\,i\,\bar{j}_i}\geq 0$,
\[
f(x^{\,\infty}) - g(x^{\,\infty})  \,\leq \, f(x) - [\,g(x^{\,\infty}) + \nabla g(x^{\,\infty})^T(x-x^{\,\infty})\,],
\]
which, by Lemma \ref{lemma:weak B},  yields that $x^{\,\infty}$ is an A-stationary point of the problem \eqref{eq: dc2}.
\end{proof}

\subsection{Solving the subproblem of the dc algorithm}
\label{sec:QP sub}
Given $\bar{z} \triangleq (\bar{w}, \bar{\beta}, \bar{\sigma}^\pm)$ and a positive constant $c > 0$, the strongly convex objective of the subproblem of the dc algorithm in Step 1 for solving the problem \eqref{eq:OCE optimization problem 3} with $u$, $P_a$ and $P_b$ given in \eqref{eq:utility and sparsity} can be essentially written as
\[\small
\begin{array}{l}
\lambda_a^N \, \displaystyle{
	\sum_{i=1}^p
} \, \left[ \, \phi_i^a | \, a_{\,i} \, | - \displaystyle{
	\frac{d \rho_i^a(\bar{a}_i)}{da_i}
} \, ( \, a_i - \bar{a}_i \, ) \, \right] + \lambda_{\beta}^N \, \displaystyle{
	\sum_{i=1}^p
} \, \left[ \, \phi_i^{\beta} | \, \beta_{\,i} \, | - \displaystyle{
	\frac{d \rho_i^{\beta}(\bar{\beta}_i)}{d\beta_i}
} \, ( \, \beta_i - \bar{\beta}_i \, ) \, \right] +\\[0.2in] \displaystyle{
	\frac{1}{N}
} \, \displaystyle{
	\sum_{i=1}^N
} \,   \displaystyle{
	\frac{{\cal Z}^-_i \sigma^-_i}{\pi(A_{\,i} \,|\, X^{\,i})}
} -  \displaystyle{
	\frac{1}{|\calN_+|}
} \displaystyle{
	\sum_{i \in \calN_+}
} \, \displaystyle{
	\frac{{\cal Z}^+_i \sigma^+_i}{\pi(A_{\,i} \,|\, X^{\,i})}
}  + \displaystyle{ \frac{c}{2}} \, \displaystyle{  ||z - \bar{z}||^2} + \\ [0.2in]
\displaystyle{
	\frac{1}{2 \, \pi(A_{\,i} \,|\, X^{\,i})}
} \,  \, \bigg\{(1-\xi_1)\left[ \, \max(0,t_i) + \sigma_i^- \, \right]^2 + (1+\xi_2)\left[ \, \max(0,-t_i) + \sigma_i^- \, \right]^2 - \\[0.2in]

2(2-\xi_1+\xi_2)\bar{\sigma}^-_i(\sigma_i^- - \bar{\sigma}^-_i) -
2\left[\,(1-\xi_1)\max(0,\bar{t}_i) - (1+\xi_2)\max(0,-\bar{t}_i)\,\right]\,(t_i - \bar{t_i})\bigg\},
\end{array}
\]
where $z \triangleq ( w,\beta,\sigma^\pm )$ with $\beta \in \mathbb{R}^{p}$, $w \triangleq ( a,b ) \in S$,  $\sigma^- \in \mathbb{R}^{N}$ and $\sigma^+ \in \mathbb{R}^{|\calN_+|}$.
The above objective function involves the convex, non-differentiable terms $| a_i |$, $| \beta_i |$, $\left[ \, \max(0,t_i) + \sigma_i^- \, \right]^2$,
and $\left[ \, \max(0,-t_i) + \sigma_i^- \, \right]^2$; the latter two squared
terms also make the objective non-separable in the $w$ and $\sigma^-$ variables. All these features make the linear inequality constrained
subproblem seemingly complicated. One way to solve this subproblem is via the dual semismooth Newton approach, as discussed in a recent paper \cite{cui2018composite}.
In fact, by introducing auxiliary variables
\[\left\{\begin{array}{ll}
t_i^+ \, = \, \max(t_i, 0), & t_i^- \, = \, \max(-t_i, 0),\\[0.1in]
a_i^+ \, = \, \max(a_i, 0), & a_i^- \, = \, \max(-a_i, 0),\\[0.1in]
b_i^+ \, = \, \max(b_i, 0), & b_i^- \, = \, \max(-b_i, 0),
  \end{array}\right.
  \]
we may write
\[
{\cal Z}_i - w^{\, T} \wh{X}^{\, i} \, = \, t_i = t_i^+ - t_i^-, \epc |a_i| = a_i^+ + a_i^-, \epc |\beta_i| = \beta_i^+ + \beta_i^-. \]
Therefore,  an alternative approach for solving \eqref{eq:OCE optimization problem 3} is to  transform it into a
standard quadratic programming  problem with the additional variables $(t_i^+, t_i^-, a_i^+, a_i^-,b_i^+, b_i^-)$
such that  it can be solved by many efficient quadratic programming solvers.

{In terms of statistical consistency, as long as the tuning parameters $\lambda_a^{\, N}$ and $\lambda_\beta^{\, N}$ go to $0$ when $N$ goes to infinite, the minimizer of the empirical objective function \eqref{eq:OCE binaryoptimization problem} might converge to the minimizer of the corresponding population problem under some regularity conditions (\cite{van2000asymptotic}). If we allow rates of tunning parameters going to 0 faster than $\frac{1}{\sqrt{n}}$, then the convergence rate of empirical minimizers may be $\frac{1}{\sqrt{n}}$ under some regularity conditions. Similar ideas could be borrowed from \cite{knight2000asymptotics}, although their considered settings are different from ours. The convergence results in our settings are more complicated than those standard cases since the empirical loss function here is non-convex and non-smooth.
}

\section{Numerical Experiments}
\label{sec:Numerical}
In this section, we demonstrate the effectiveness of the proposed IDR-CDE in finding optimal IDRs via three synthetic examples.
The subproblem of the dc algorithm, being equivalent to a quadratic programming problem, is solved by the commercial solver Gurobi
with an academic license.  All the numerical results are run in Matlab on Mac OS X with 2.5 GHz Intel Core i7 and 16 GB RAM.
We use piecewise linear affine function given by \eqref{CVaR} with $\xi_1 = 0, \xi_2 = 0.5$ in all the experiments, which is
equivalent to estimating the optimal IDR that maximizes $\text{CVaR}_{0.5}(\calZ)$.  In practice, users can decide their own utility
functions and values $\xi_1$, $\xi_2$ based on the specific problem settings.  If one believes there may have high risks for inappropriate
decisions and wants to control the risk of higher-risk individuals, it would be better to use robust utility functions such as the
piecewise affine utility function.  We consider a binary-action space in a randomized study with $\pi(A_{\,i} = \pm 1 \,|\, X_i) = 0.5$.
All the tuning parameters such as $\lambda_b^N$ and $\lambda_{\beta}^N$ are selected via $10$-fold-cross-validation that maximizes
the following average of the empirical ${\cal O}_{(u, \calF)}^{\,d}(\calZ)$, which is defined as
\[
\widehat{\cal O}_{(u, \calF)}^{\, \wh{d}}(\calZ) \,\triangleq \, \displaystyle{
\frac{\displaystyle{
\sum_{i \in \calN}
} \, \left[ \, \wh{\alpha}(X_i) + u(\calZ_i - \wh{\alpha}(X_i)) \, \right] \, \displaystyle{
\frac{\II(A_i = \wh{d}(X_i))}{\pi(A_i | X_i)}}}{\displaystyle{
\sum_{i \in \calN}
} \, \displaystyle{
\frac{\II(A_i = \wh{d}(X_i))}{\pi(A_i | X_i)}
}}} \, .
\]
Specifically, we divide the training data into 10 groups.  For each fold, we estimate the optimal IDR $\wh{d}(X)$ using 9 groups
of the data (the training set) for a pre-specified series of tuning parameters $\lambda_b^N$ and $\lambda_{\beta}^N$ and then compute
$\widehat{\cal O}_{(u, \calF)}^{\, {d}}(\calZ)$ on the remaining group of data (the test set).  The best tuning parameters are the
ones that lead to the largest values of $\widehat{\cal O}_{(u, \calF)} ^{\, \wh{d}}(\calZ)$.  The so-obtained parameters are then
employed to re-compute the optimal IDR using the entire set of data.

We compare our approach with three existing methods under the expected-value function framework $\E^{\, d}[\calZ]$.
The first one is a model-based method called $l_1$-PLS \cite{qian2011performance} that first fits a penalized least-square regression
with covariate function $(1, X, A, X \circ A)$ on $\calZ$ to estimate $\E[\calZ | X, A = a]$, and then select the action with the largest
$\E\,[\,\calZ \,|\, X, A = a\,]$, where $X \circ A$ denotes the element-wise product.  The second one is a classification-based method
called residual weighted learning (RWL)  \cite{zhou2017residual} that consists of two steps: (1) fitting a least-square regression
on $\calZ_i$ with covariates $\wh{X}_i$ to compute the residual $r_i$ for each data point in order to remove the main effect;
(2) applying the support vector machine with truncated loss to compute the optimal IDR with each data point weighted by $r_i$.
The third one is the direct learning (DLearn) method \cite{qi2017} that lies between the model-based and the classification-based method,
where the optimal IDR is directly found by weighted penalized least square regression on $\calZ A$ with covariates $\wh{X}$, based on the
fact that \[\E\,[\,\calZ \,|\, X, A = 1\,] - \E\,[\,\calZ \,|\, X, A = -1\,] = \E\,\left[ \, \displaystyle{
\frac{\calZ\, A}{\pi(A | X)}
} \, | \, X \, \right].\]

The simulation data are generated by the  model
\begin{equation*}
\calZ = m(X) + h(X)A + \varepsilon,
\end{equation*}
where $m(X)$ is the main effect, $h(X)$ is the interaction effect with treatment $A$, and $\varepsilon$ is the random error.
We consider the same main effect and interaction effect functions: $m(X) =1+X_1+X_2$ and $h(X) = 0.5+X_1-X_2+X_3$ respectively,
but various types of {asymmetric} error distributions under three simulation scenarios:\\[0.05in]
(1) $\log(\varepsilon)$ follows a normal distribution with mean 0 and standard deviation 2;\\[0.05in]
(2) the random error $\varepsilon$ follows a Weibull distribution with scale parameter $0.5$ and shape parameter $0.3$;\\[0.05in]
(3) $\log(\varepsilon)$ follows a normal distribution with mean 0 and standard deviation $2|1+X_1+X_2|$.

 The above scenarios address heavy right tail distributions to test the robustness of different methods.
In particular, the log-normal distribution is frequently used in the finance area, the Weibull distribution is commonly considered
in survival analysis of clinical trials, and the  third scenario considers a heterogeneous error distribution depending on covariates. {In all our simulation studies, the error distributions are asymmetric}.

The training sample size is set to be 100 and 200, and the number of covariates $p$ is fixed to be $10$.  Each covariate is
generated by uniform distribution on $[-1,1]$.  In Table~\ref{tab:time}, we list the average computational time and the iteration
numbers of the dc algorithm for solving the problem  \eqref{eq:OCE optimization problem 3} with $\lambda_a^N=0.1$ and  $\lambda_{\beta}^N=0.1$
over 100 simulations. One can see that the proposed algorithm is very efficient and robust for solving the empirical IDR problem.

\begin{table}[h]
	\centering
	\small
	\scalebox{1}{
		\begin{tabular}{@{}lccccc@{}}
			\addlinespace
			\toprule
			& \multicolumn{2}{c}{$n=100$} & \phantom{a} & \multicolumn{2}{c}{$n=200$} \\
			\cmidrule{2-3} \cmidrule{5-6}
			& time & iteration numbers  &  &  time & iteration numbers \\
			\midrule
			Scenario 1 & 0.70 & 18 && 2.10 & 20 \\
			\midrule
			Scenario 2 & 0.79& 18 && 2.08 & 20 \\
			\midrule
			Scenario 3 & 0.68 & 16  && 1.88 & 18 \\
			\bottomrule
		\end{tabular}
	}
	\caption{\small The average computational times (in seconds) and dc iteration numbers for $p = 10$.}
	\label{tab:time}
\end{table}

The comparisons of the four methods for finding optimal IDRs over 100 replications are based on the following four criteria:\\[0.05in]
(1) the misclassification error rate on the test data (this is possible since the optimal IDR under our simulation settings is known,
which is $\sign(0.5 + X_1-X_2+X_3)$);\\[0.05in]
(2) the empirical average of outcome under the decision rule over test data, which is defined as
\[
\widehat{\E}^{\, d} \, \left[ \, \calZ \, \right] \, = \, \displaystyle{
\frac{\displaystyle{
\sum_{i \in \calN_1}
} \, \displaystyle{
\frac{\calZ_i\, \II(A_i = \wh{d}(X_i))}{\pi(A_i | X_i)}
}}{\displaystyle{
\sum_{i \in \calN_1}
} \, \displaystyle{
\frac{\II(A_i = \wh{d}(X_i))}{\pi(A_i | X_i)}
}}},
\]
where $\calN_1$ is the index of test data set. This value evaluates the expected outcome of $\calZ$ if the action assignment
follows the estimated decision rules $\widehat{d}(X)$;\\[0.05in]
(3) the empirical $50\%$  quantile of $\calZ_i\II(A_i = \wh{d}(X_i))$ on the test data;\\[0.05in]
(4) the empirical  $25\%$ quantiles of $\calZ_i\II(A_i = \wh{d}(X_i))$ on the test data.\\[0.1in]
The test data in each scenario are independently generated with size 10,000.

\begin{table}[h]
	\centering
	\small
	\scalebox{1}{
		\begin{tabular}{@{}lccccc@{}}
			\addlinespace
			\toprule
			& \multicolumn{2}{c}{$n=100$} & \phantom{a} & \multicolumn{2}{c}{$n=200$} \\
			\cmidrule{2-3} \cmidrule{5-6}
			& Misclass. & Value  &  & Misclass. & Value \\
			\midrule
			\multicolumn{6}{c}{Scenario 1} \\
			\midrule
			DLearn & 0.48(0.02) & 8.36(0.09) && 0.47(0.02) & 8.5(0.07)\\
			$l_1$-PLS & 0.45(0.01) & 8.46(0.06) && 0.45(0.01) & 8.58(0.09)\\
			RWL & 0.42(0.01) & 8.53(0.07) && 0.42(0.01) & 8.59(0.07)\\
			IDR-CDE & \bf{0.25}(0.01) & \bf{8.98}(0.07) && \bf{0.17}(0.01) & \bf{9.15}(0.08)\\
			\midrule
			\multicolumn{6}{c}{Scenario 2} \\
			\midrule
			DLearn & 0.44(0.02) & 5.82(0.06) && 0.44(0.02) & 5.74(0.06)\\
			$l_1$-PLS & 0.42(0.01) & 5.89(0.05) && 0.4(0.01) & 5.86(0.05)\\
			RWL & 0.39(0.01) & 5.95(0.04) && 0.37(0.01) & 5.96(0.04)\\
			IDR-CDE & \bf{0.21}(0.01) & \bf{6.36}(0.04) && \bf{0.15}(0.01) & \bf{6.41}(0.04)\\
			\midrule
			\multicolumn{6}{c}{Scenario 3} \\
			\midrule
			DLearn & 0.5(0.02) & 3948.04(659.88) && 0.51(0.02) & \bf{26588.55}(13692.58)\\
			$l_1$-PLS & 0.48(0.01) & \bf{4758.49}(801.06) && 0.5(0.01) & 26209.19(13702.62)\\
			RWL & 0.48(0.01) & 4256.27(774.97) && 0.47(0.01) & 24463.43(13592.7)\\
			IDR-CDE & \bf{0.24}(0.01) & 4113.85(934.74) && \bf{0.2}(0.01) & 25712.22(13473.72)\\
			\bottomrule
		\end{tabular}
	}
	\caption{\small Average misclassification rates (standard errors) and  average means (standard errors) of empirical value functions
for three simulation scenarios over 100 runs. The best expected value functions and the minimum misclassification rates are in bold.}
	\label{tab:p50nonlinear}
\end{table}
\begin{table}[h]
	\centering
	\small
	\scalebox{1}{
		\begin{tabular}{@{}lccccc@{}}
			\addlinespace
			\toprule
			& \multicolumn{2}{c}{$n=100$} & \phantom{a} & \multicolumn{2}{c}{$n=200$} \\
			\cmidrule{2-3} \cmidrule{5-6}
			& $50\%$ quantile & $25\%$ quantile  &  & $50\%$ quantile & $25\%$ quantile\\
			\midrule
			\multicolumn{6}{c}{Scenario 1} \\
			\midrule
			DLearn & 2.64(0.04) & 1.17(0.04) && 2.67(0.04) & 1.21(0.05)\\
			$l_1$-PLS & 2.73(0.03) & 1.26(0.03) && 2.74(0.03) & 1.25(0.03)\\
			RWL & 2.81(0.03) & 1.35(0.03) && 2.83(0.03) & 1.35(0.04)\\
			IDR-CDE & \bf{3.17}(0.01) & \bf{1.81}(0.02) && \bf{3.26}(0.01) & \bf{1.99}(0.01)\\
			\midrule
			\multicolumn{6}{c}{Scenario 2} \\
			\midrule
			DLearn & 1.96(0.04) & 0.69(0.04) && 1.97(0.04) & 0.7(0.05)\\
			$l_1$-PLS & 2.01(0.03) & 0.77(0.03) && 2.08(0.03) & 0.82(0.03)\\
			RWL & 2.1(0.03) & 0.85(0.03) && 2.16(0.03) & 0.92(0.03)\\
			IDR-CDE & \bf{2.47}(0.01) & \bf{1.36}(0.02) && \bf{2.53}(0.01) & \bf{1.47}(0.01)\\
			\midrule
			\multicolumn{6}{c}{Scenario 3} \\
			\midrule
			DLearn & 2.22(0.05) & 1.02(0.05) && 2.2(0.05) & 1.01(0.05)\\
			$l_1$-PLS & 2.3(0.03) & 1.04(0.03) && 2.24(0.03) & 0.99(0.03)\\
			RWL & 2.29(0.03) & 1.07(0.03) && 2.31(0.03) & 1.09(0.03)\\
			IDR-CDE & \bf{2.81}(0.01) & \bf{1.73}(0.01) && \bf{2.86}(0.01) & \bf{1.8}(0.02)\\
			\bottomrule
		\end{tabular}
		}
		\caption{\small Results of average $25\%$ (standard errors) and  $50\%$ (standard errors) quantiles of empirical value functions
for three simulation scenarios over 100 runs. The largest $25\%$ and $50\%$ quantiles are in bold.}
	\label{tab:p10quantiles}
\end{table}

Several observations can be drawn from these simulation examples in Tables \ref{tab:p50nonlinear} and \ref{tab:p10quantiles}.  First of all, our method under the IDR-CDE has the smallest classification
error in choosing correct decisions compared with those under the criterion of expected outcome.  Under the piecewise utility function, we
emphasize more on improving subjects with relative low outcome, in contrast to focusing on average, which may ignore the subjects with higher-risk.
As a result, in addition to misclassification rate, the $50\%$ and $25\%$ quantiles of expected-value functions are also the largest among all
the methods.  Secondly, the advantages of our method become more obvious if comparing the $25\%$ quantiles of the empirical value functions
on the test data with $50\%$ quantiles.  For example, in the second scenario, the $25\%$ quantiles of empirical value functions of our method
 are almost twice as large as those by DLearn.  Another interesting finding is that in the last scenario, although the average empirical
 value functions of $l_1$-PLS and RWL are larger than those of our method, our method is indeed much better based on the misclassification
 error and the quantiles.  One possible reason is that these methods under the expected value function framework only correctly identify
 the decisions for subjects in lower risk while ignoring subjects with potentially higher risk. The estimated optimal IDRs by those methods
 may lead to serious problems, especially in precision medicine when assigning treatments to patients.
 Although, on average, patients may gain benefits of following those decision rules, some patients may come across high risk,
 causing adverse events such as exacerbation in practice by using the recommended treatment using the standard criterion of expected outcome.

 {In terms of real data applications, there are several possibilities. For example, we can use the piecewise linear utility function 
 to control the lower tails of outcomes for individual patient in AIDS or cancer studies. Another potential application is to use the quadratic utility 
 function to take variance of each decision rules into consideration. The performance of the results by our method depends on the choice of the covariate-dependent $\alpha(X)$ and the utility function $u$. We leave these as the future work.}
\section{Acknowledgements}
\label{sec:Acknowledgements}
The authors thank two referees and the associate editor for their careful reading of the paper and for the comments that have helped to improve the quality of this paper.

\bibliographystyle{siamplain}
\bibliography{references}

\end{document}